\documentclass[12pt,a4paper,psamsfonts]{amsart}
\usepackage{fancyhdr}
\usepackage{appendix}
\usepackage{amssymb,amscd,amsxtra,calc}
\usepackage{mathrsfs}
\usepackage{cmmib57}
\usepackage{multirow}
\usepackage[all]{xy}
\usepackage[colorlinks,linkcolor=blue,anchorcolor=blue,citecolor=green]{hyperref}
\theoremstyle{plain}
    \newtheorem{thm}{Theorem}[section]

    \newtheorem{corollary}[thm]{Corollary}
    
    \newtheorem{lemma}[thm]{Lemma}
    \newtheorem{proposition}[thm]{Proposition}
    \newtheorem{question}[thm]{Question}
    \newtheorem{theorem}[thm]{Theorem}

\theoremstyle{definition}

    \newtheorem{remark}[thm]{Remark}
\theoremstyle{remark}

    \newtheorem{setup}[thm]{}

\newcommand{\BCC}{\mathbb{C}}

\newcommand{\Q}{\mathbb{Q}}

\newcommand{\R}{\mathbb{R}}

\newcommand{\Z}{\mathbb{Z}}

\newcommand{\alb}{\operatorname{alb}}

\newcommand{\Aut}{\operatorname{Aut}}

\newcommand{\Gal}{\operatorname{Gal}}

\newcommand{\id}{\operatorname{id}}

\newcommand{\Ker}{\operatorname{Ker}}

\newcommand{\NS}{\operatorname{NS}}

\newcommand{\PEC}{\operatorname{PEC}}

\newcommand{\Stab}{\operatorname{Stab}}
\newcommand{\Sing}{\operatorname{Sing}}

\newcommand{\Supp}{\operatorname{Supp}}

\newcommand{\Nlc}{\operatorname{Nlc}}
\newcommand{\Nklt}{\operatorname{Nklt}}

\newcommand{\sep}{\operatorname{sep}}

\newcommand{\Codim}{\operatorname{codim}}

\newcommand{\N}{\operatorname{N}}

\newcommand{\Alb}{\operatorname{Alb}}

\newcommand{\SSpec}{\operatorname{\mathit{Spec}}}

\newcommand{\Pic}{\operatorname{Pic}}

\newcommand{\red}{\mathrm{red}}

\newcommand{\reg}{\mathrm{reg}}

\begin{document}

\title[Polarized endomorphisms in arbitrary characteristic]
{Polarized endomorphisms of normal projective threefolds in arbitrary characteristic
}

\author{Paolo Cascini, Sheng Meng, and De-Qi Zhang}

\address
{
\textsc{Department of Mathematics} \endgraf
\textsc{Imperial College London, London SW72A2, United Kingdom
}}
\email{p.cascini@imperial.ac.uk}
\address
{
\textsc{Department of Mathematics} \endgraf
\textsc{National University of Singapore,
Singapore 119076, Republic of Singapore
}}
\email{ms@u.nus.edu}
\address
{
\textsc{Department of Mathematics} \endgraf
\textsc{National University of Singapore,
Singapore 119076, Republic of Singapore
}}
\email{matzdq@nus.edu.sg}

\begin{abstract}
Let $X$ be a projective variety over an algebraically closed field $k$ of arbitrary characteristic $p \ge 0$.
A surjective endomorphism $f$ of $X$ is $q$-polarized if
$f^\ast H \sim qH$ for some ample Cartier divisor $H$ and integer $q > 1$.

Suppose $f$ is separable and $X$ is $\Q$-Gorenstein and normal.
We show that the anti-canonical divisor $-K_X$ is numerically equivalent to an effective
$\Q$-Cartier
divisor,
strengthening slightly the conclusion of Boucksom, de Fernex and Favre \cite[Theorem C]{BFF}
and also covering singular varieties over an algebraically closed field of arbitrary characteristic.

Suppose $f$ is separable and $X$ is normal. We show that the Albanese morphism of $X$ is an algebraic fibre space
and $f$ induces polarized endomorphisms
on the Albanese and also the Picard variety of $X$, and $K_X$ being pseudo-effective and
$\Q$-Cartier means being a torsion $\Q$-divisor.

Let $f^{\Gal}:\overline{X}\to X$ be the Galois closure of $f$.
We show that if $p>5$ and co-prime to $\deg f^{\Gal}$ then one can run the minimal model program (MMP) $f$-equivariantly, after replacing $f$ by a positive power, for a mildly singular threefold $X$
and reach a variety $Y$ with torsion canonical divisor
(and also with $Y$ being a quasi-\'etale quotient of an abelian variety when $\dim(Y)\le 2$).
Along the way, we show that
a power of $f$ acts as a scalar multiplication on the Neron-Severi group of $X$ (modulo torsion)
when $X$ is a smooth and rationally chain connected projective variety of dimension at most three.
\end{abstract}



\maketitle
\tableofcontents

\section{Introduction}
We work over a fixed algebraically closed field $k$ of arbitrary characteristic $p \ge 0$.

Let $X$ be a projective variety over the field $k$ and $f:X\to X$ a surjective endomorphism.
We say that $f$ is {\it $q$-polarized}  (resp. {\it numerically $q$-polarized}) by $H$,
if there is an ample Cartier divisor $H$ such that $f^{\ast}H \sim qH$, a linear equivalence
(resp. $f^{\ast}H \equiv qH$, a numerical equivalence)
for some integer $q>1$.
When $char\, k=0$, being numerically polarized is equivalent to being polarized after replacing
$H$ (cf.~\cite[Lemma 2.3]{Na-Zh}).  In Section \ref{sec-app}, Theorem \ref{thm-num-lin},
we generalize this result to arbitrary characteristic if we further assume $X$ is normal and $f$ is separable.

Let $X$ be a $\Q$-Gorenstein normal projective variety over the field $k$
admitting a polarized separable endomorphism.
Boucksom, de Fernex and Favre \cite[Theorem C]{BFF} showed that if $X$ is further smooth and $char \, k = 0$, then $-K_X$ is pseudo-effective.
Our first Theorem \ref{thm-pe} below
strengthens a bit their conclusion
and also covers singular varieties over the field $k$ of arbitrary characteristic.

\begin{theorem}\label{thm-pe}
Let $f: X \to X$ be a polarized separable endomorphism of a $\Q$-Gorenstein normal projective variety $X$
of dimension $n \ge 0$ over the field $k$ of arbitrary characteristic.
Then
$-K_X$ is numerically equivalent to an effective
$\Q$-Cartier divisor.
In particular, if $\Alb(X)$ is trivial (cf.~Section \ref{sec-app}),
then we have the Iitaka $D$-dimension $\kappa(X,-K_X)\ge 0$.
\end{theorem}

Our next Theorem \ref{ap-cor-dual} below affirmatively answers Krieger - Reschke \cite[Question 1.10]{Kr} and generalizes \cite[Corollary 1.4]{MZ} to all characteristics, the proof of which is very different, due to the lack
of practical characterizations of abelian varieties in positive characteristics.
It is also used in the proof of Theorems \ref{thm-torsion} and \ref{scalarthm}.

\begin{theorem}\label{ap-cor-dual} Let $f:X\to X$ be a numerically $q$-polarized separable endomorphism of a normal projective variety $X$. Then we have the following.
\begin{itemize}
\item[(1)]
The Albanese morphism (cf.~Section \ref{sec-app}) $\alb_X:X\to \Alb(X)$ is surjective with $(\alb_X)_\ast \mathcal{O}_X=\mathcal{O}_{\Alb(X)}$ and all the fibres of $\alb_X$ are irreducible and equi-dimensional. The induced morphism $g:\Alb(X)\to \Alb(X)$ is $q$-polarized separable.
\item[(2)]
The Albanese map (cf.~Section \ref{sec-app}) $\mathfrak{alb}_X:X\dashrightarrow \mathfrak{Alb}(X)$ is dominant and the induced morphism $h:\mathfrak{Alb}(X)\to \mathfrak{Alb}(X)$ is $q$-polarized separable.
\item[(3)]
The dual $\widehat{f}:\Pic^0(X)_{\red} \to \Pic^0(X)_{\red}$ is numerically $q$-polarized.
\end{itemize}
\end{theorem}

\begin{question}\label{question1}
Let $f:X\to X$ be a polarized separable endomorphism of a normal projective variety $X$. Does $K_X$ being pseudo-effective imply $X$ being $Q$-abelian?
\end{question}

Question \ref{question1} above was just \cite[Conjecture 1.2]{Na-Zh}. In positive characteristic 2 or 3,
there are examples $X$, due to L. Moret-Bailly, such that each $X$ is a smooth projective surface or threefold
with all $\ell$-adic Chern classes $c_i(X)$ vanishing but with $X$ not being a $Q$-abelian variety; we
do not know whether these $X$ admit separable polarized endomorphisms or not;
see \cite[\S 7.3]{Lan} for details and references.

Question \ref{question1} has a positive answer when $\dim(X) = 2$, char $k > 5$ and $p\nmid\deg f^{\Gal}$ (cf. Theorem \ref{thm-Q-surf}), or when $X$ is a klt projective variety over $\BCC$, by first showing the vanishing of the Chern classes
$c_i(X)$, $i = 1, 2$, and then appealing to a result generalizing Yau's characterization of $Q$-abelian varieties
in terms of the vanishing of the first two Chern classes
(cf.~Lemma \ref{lem-pe-qe}, Theorem \ref{thm-torsion}, \cite[Theorem 3.4]{Na-Zh},
\cite[Theorem 1.21]{GKP} and \cite[Lemma 4.6]{MZ}).
Such characterization does not hold in positive characteristic (cf. \cite{Lan}).

We are not able to answer Question \ref{question1} in higher dimension. But we have:

\begin{theorem}\label{thm-torsion}
Let $f:X\to X$ be a numerically polarized separable endomorphism of a normal projective variety $X$ with $K_X$ being pseudo-effective and $\Q$-Cartier. Then $f$ is quasi-\'etale and $K_X\sim_{\Q}0$.
\end{theorem}

With the above results, we are ready to run an $f$-equivariant minimal model program (MMP)
over the field $k$ of arbitrary characteristic.
We refer to \cite{KM} for the definitions of {\it canonical}, {\it klt} or {\it lc} singularities,
and \cite{Hara} for the definition of {\it strongly F-regular} singularity.
In the case of characteristic $0$,
one can run the MMP $f$-equivariantly for mildly singular $X$ in any dimension
and reach a $Q$-{\it abelian variety} (i.e.~ a quasi-\'etale quotient of an abelian variety; see \ref{nat}) or a {\it Fano variety} (i.e.~the anti-canonical divisor $-K_X$ is ample) of Picard number one
(cf.~\cite[Theorem 1.8]{MZ}).
In the case of positive characteristic, a surjective endomorphism might be ramified everywhere.
So we need to consider the separable endomorphisms in order to apply the ramification divisor formula which is crucial and automatically satisfied in the case of characteristic 0. Another problem in the case of positive characteristic is that the MMP is not fully established (only known for lc 3-fold with characteristic $>5$, cf.~\cite{Wal} and the references therein).
Moreover, after a Fano contraction, we want the base variety still to have mild singularities so that we can further run the MMP on the base variety; this is taken care of by Theorems
\ref{thm-Q-surf} and \ref{thm-w}.

In the case of characteristic $0$, Nakayama \cite[Section 7.3]{ENS}
showed that a normal projective surface $X$ having $K_X\sim_{\Q} 0$
and a quasi-\'etale non-isomorphic endomorphism $f$, is $Q$-abelian.
In particular, $X$ is klt and $\Q$-factorial.
In positive characteristic,
we are unable to show that $X$ is lc even when $f$ is separable, since $f$ may have wild ramification.
But when $p\nmid \deg f^{\Gal}$ and $p > 5$, we have the same result below.
Here for a separable finite surjective morphism $h: X_1 \to X_2$,
denote by $h^{\Gal}:\overline{X}_1\to X_2$ its {\it Galois closure}.

\begin{theorem}\label{thm-Q-surf}
Let $f:X\to X$ be a polarized endomorphism of
a normal projective surface $X$ over the field $k$ of characteristic $p>5$.
Suppose $p\nmid\deg f^{\Gal}$ and $K_X$ is pseudo-effective.
Then $X$ is a $Q$-abelian surface. In particular, $X$ is $\Q$-factorial and klt.
\end{theorem}

To state the remaining results, we need the following hypothesis.
\vskip 2mm
\noindent\textbf{Hyp(A).}
Fix a $\mathbb{Q}$-factorial klt normal projective variety $X$ of $\dim(X)\le 3$ over the field $k$ with characteristic $p>5$,
and a $q$-polarized endomorphism $f:X\to X$ with $p$ and $q$ being co-prime.

\vskip 2mm
The assumption above that $p$ and $q$ are co-prime is equivalent to
that $p$ and $\deg f = q^{\dim(X)}$ are co-prime, and cannot be simply weakened to
that $f$ is separable in our arguments of showing that the extremal rays
in the MMP are $f$-periodic. This is because a separable map $f$
may not restrict to a separable map on a subvariety; see Remark \ref{rmk-sep}.
The assumption $p > 5$ is needed in order to run the MMP and
apply Hara's result \cite{Hara}: klt surface singularities are strongly $F$-regular.

\begin{theorem}\label{scalarthm}
Assume $f:X\to X$ satisfies the hypothesis \textbf{Hyp(A)}.
Then, replacing $f$ by a positive power, there is an $f$-equivariant MMP
$$X=X_1\dashrightarrow \cdots \dashrightarrow X_i \dashrightarrow \cdots \dashrightarrow X_r=Y$$ (i.e. $f=f_1$ descends to $f_i$ on each $X_i$), with every $X_i \dashrightarrow X_{i+1}$ a divisorial contraction, a flip or a Fano contraction, of a $K_{X_i}$-negative extremal ray, such that we have:
\begin{itemize}
\item[(1)] $K_Y\sim_{\Q} 0$ and $f_r$ is quasi-\'etale.
\item[(2)]
If $K_X$ is pseudo-effective, then $X=Y$.
\item[(3)]
For each $i$, $f_i$ is a $q$-polarized endomorphism.
\item[(4)]
$f^* \, |_{\N^1(X)}$ is a scalar multiplication: $f^* \, |_{\N^1(X)} = q \, \id$, if and only if so  is $f_{r}^\ast|_{\N^1(Y)}$.
\end{itemize}
Suppose further $p\nmid \deg f^{\Gal}$.
Then we have:
\begin{itemize}
\item[(5)]
If $K_X$ is not pseudo-effective, then $Y$ is $Q$-abelian; so the MMP is relative over $Y$.
\item[(6)]
For each $i$, $X_i$ is $\Q$-factorial lc, and $X_i\dashrightarrow Y$ is an equi-dimensional
well-defined morphism with every fibre irreducible.
\item[(7)]
If $X_i \dashrightarrow X_{i+1}$ is birational for some $i$, then $\dim(Y) \le \dim(X_i) - 2$ and $f^* \, |_{\N^1(X)}$ is a scalar multiplication.
\item[(8)]
If $\dim(Y)>0$, then $X_i$ is klt for each $i$.
\end{itemize}
\end{theorem}

\begin{theorem}\label{thm-cases} Assume $f:X\to X$ satisfies the hypothesis \textbf{Hyp(A)} with $p\nmid\deg f^{\Gal}$ and $Y$ the end product of the MMP in Theorem \ref{scalarthm}.
Then one of the following occurs.
\begin{itemize}
\item[(1)] $X=Y$ and $\dim(X)\ge 2$, $K_X\sim_{\Q}0$, and $f$ is quasi-\'etale.
\item[(2)] $\dim(Y)\le 1$. So $f^* \, |_{\N^1(X)} = q \, \id$, after replacing f by a positive power.
\item[(3)] The $X \dasharrow Y$ in Theorem \ref{scalarthm} is a Fano contraction of a $K_X$-negative extremal ray.
\end{itemize}
If $\dim(Y)>0$,
then the \'etale fundamental group $\pi_1^{\acute{e}t}(X_{\reg})$ of the smooth part $X_{\reg}$ of $X$ is infinite.
\end{theorem}

Let $X$ be a normal projective variety which is rationally chain connected; see
\cite[Definition 4.21]{De}.
It is known that $\Alb(X)$ is trivial and the \'etale fundamental group of such $X$ is finite (cf.~\cite[Theorem 1.6]{Ko93}).
This, together with Theorems \ref{thm-pe}, \ref{scalarthm} and \ref{thm-cases}, imply the following.

\begin{theorem}\label{thm-smooth-rc}  Assume $f:X\to X$ satisfies the hypothesis \textbf{Hyp(A)} and $p\nmid\deg f^{\Gal}$.
Assume further that $X$ is smooth and rationally chain connected.
Then we have:
\begin{itemize}
\item[(1)]
$\kappa(X, -K_X)\ge 0$.
\item[(2)]
The MMP in Theorem \ref{scalarthm} ends up with a point.
\item[(3)] Replacing $f$ by a positive power, $f^\ast|_{\N^1(X)} = q \, \id$.
\end{itemize}
\end{theorem}

\par \vskip 1pc
{\bf Acknowledgement.}
The last author would like to thank Imperial College London, for the warm hospitality,
and Hiromu Tanaka for the fruitful discussions in May - June 2017;
he is supported by an ARF of NUS.
The second author would like to thank Max Planck Institute for Mathematics, Bonn, for the
support, as a Post Doctor Fellow.
The authors would like to thank Adrian Langer for the enlightening discussions and reference \cite{Lan}, and the referee for the careful reading and valuable suggestions.

\section{Preliminary results}

\begin{setup}{\bf Notation and terminology}\label{nat}.

{\rm
Let $X$ be a projective variety of dimension $n$ over the field $k$.
We use Cartier divisor $H$ (always meaning integral, unless otherwise indicated)
and its corresponding invertible sheaf $\mathcal{O}(H)$ interchangeably.
Denote by $\N^1(X):=\NS(X) \otimes_{\Z} \mathbb{R}$, where $\NS(X)$ is the {\it N\'eron-Severi group}.
Note that $\N^1(X)$ can also be regarded as the quotient vector space of $\R$-Cartier divisors modulo the numerical equivalence; see e.g. \cite[Definition 2.1]{MZ} for the definition of numerical equivalence ``$\equiv$" of $\R$-Cartier divisors.
Denote by $\PEC(X)$ the closure of the set of classes of effective $\mathbb{R}$-Cartier divisors in $\N^1(X)$.
An $\R$-Cartier divisor $D$ is said to be {\it pseudo-effective} if its class $[D]\in \PEC(X)$;
see e.g. \cite[Definition 2.4]{MZ} for a bit more information.

Let $f:X\to X$ be a surjective endomorphism.
A subset $Z \subseteq X$ is $f$ (resp. $f^{-1}$) {\it periodic} if
$f^s(Z) = Z$ (resp. $f^{-s}(Z) = Z$) for some $s > 0$.
We say that $f$ is {\it numerically $q$-polarized}  if there is an ample Cartier divisor $H$ such that $f^{\ast}H \equiv qH$ (numerical equivalence) for some integer $q>1$;
note that $(qH)^n = (f^{\ast}H)^n = (\deg f) H^n$ and hence
$\deg f = q^n\ge 2.$

Let $g:X\to Y$ be a finite surjective morphism of normal varieties.
We say that $g$ is {\it separable} if the induced field extension $g^* : k(Y) \, \to \, k(X)$ is separable.
Suppose $g$ is separable.
Then we have the ramification divisor formula $K_X=g^*K_Y+R_g$,
where $R_g$ is the ramification divisor of $g$ (cf.~\cite[Lemma 4.4]{Ok} or Proposition \ref{prop-rdf}).
We say $g$ is {\it quasi-\'etale} (or \'etale in codimension $1$) if $R_g=0$.
We say $Y$ is {\it $Q$-abelian} if there is a quasi-\'etale morphism $g : X \to Y$ from an abelian variety $X$.
}
\end{setup}

\begin{proposition} (cf.~\cite[Lemma 4.4]{Ok})\label{prop-rdf}
Let $f:X'\to X$ be a separable finite surjective morphism of normal varieties over the field $k$ of characteristic $p$.
Let $P$ be a prime divisor on $X'$ and let $r$ be the ramification index of $f$ along $P$.
Then there exists a non-negative integer $b\ge r-1$ such that $K_{X'}=f^*K_X+bP$ holds around the generic point of $P$ and $b>r-1$ holds exactly when $p|r$.
In particular, we have the ramification divisor formula:
$$K_{X'}=f^*K_X+R_f,$$
where $R_f:=\sum_P bP$ is the ramification divisor of $f$.
\end{proposition}

Let $Y$ be a normal birational model of $X$ and $Y'$ the normalization of $Y$ in $K(X')$ with the induced finite surjective morphism $g:Y'\to Y$.
Suppose for any $Y$, the ramification index of $g$ along any prime divisor of $Y'$ is co-prime with $p$.
We say $f:X'\to X$ is {\it tame}.
Denote by $f^{\Gal}:\overline{X}\to X$ the Galois closure of $f:X'\to X$.
If $f^{\Gal}$ is tame, then so is $f$.
Since $f^{\Gal}$ is Galois, if $p\nmid \deg f^{\Gal}$, then $f^{\Gal}$ is tame.

Let $\Delta$ be a Weil $\Q$-divisor on $X$ such that $K_X+\Delta$ is $\Q$-Cartier.
Let $\Delta'$ be a Weil $\Q$-divisor on $X$ such that $K_{X'}+\Delta'=f^*(K_X+\Delta_X)$.
Consider the following diagram:
$$\xymatrix{
Y'\ar[r]^{{\pi}'}\ar[d]^{g} &X'\ar[d]^{f}\\
Y\ar[r]^{\pi} &X }$$
where $\pi$ is a birational morphism from a normal variety $Y$ and $Y'$ is the normalization of the fibre product of $\pi$ and $f$.
Let $E$ be an exceptional prime divisor of $\pi$ and $E'$ an exceptional prime divisor of $\pi'$ which dominates $E$.
Let $r$ be the ramification index of $g$ along $E'$.
Then near the generic point of $E'$, we have $K_{Y'}=g^*K_Y+bE'$ for some integer $b\ge r-1$.
Write $\Delta = \sum \delta_i \Delta_i$ with $\Delta_i$ distinct prime divisors.
Recall that the pair $(X, \Delta)$ is
{\it subklt} (resp. {\it sublc}) if
the coefficients $\delta_i < 1$ (resp. $\delta_i \le 1$) and the
discrepancies $a(E,X,\Delta) > -1$ (resp. $\ge -1$) for all exceptional prime divisors $E$ over $X$.
A subklt (resp.~sublc) pair is {\it klt} (resp. {lc}) if further $\Delta \ge 0$; see \cite{KM}.
As in the proof of \cite[Proposition 5.20]{KM}, there is a relation between discrepancies of the two pairs:
$$a(E',X',\Delta')+1=r(a(E,X,\Delta)+1)+(b-r+1),$$
where $E$ above is a prime divisor (not necessarily exceptional) over $X$.
In the above, we have $b \ge r-1$, as mentioned early on.
Thus, we have the following:

\begin{lemma}\label{lem-km-p} Assume the above notation.
Then we have:
\begin{itemize}
\item[(1)] If $(X,\Delta)$ is subklt or sublc, then so is $(X',\Delta')$.
\item[(2)] Suppose $f$ is tame. Then the converse of (1) holds.
\end{itemize}
\end{lemma}

\begin{remark}\label{rmk-Gal}In positive characteristic, Lemma \ref{lem-km-p}(2) fails for non-tame $f$;
see \cite[Example 4.8]{Ok} for an easy counterexample.
\end{remark}

Let $E/F$ be a finite field extension. Denote by $[E:F]_s$ the separable degree and by $[E:F]_i$ the inseparable degree. The following lemma shows that the separable property is stable via the equivariant lifting of a finite cover and the equivariant descending.

\begin{lemma}\label{sepdeg} Let $j:F\to E$, $\sigma:F\to F$ and $\tau:E\to E$ be the field extensions of fields $F$ and $E$ such that $\tau\circ j=j\circ \sigma$. Suppose $\sigma$ and $\tau$ are finite field extensions and
$E$ is finitely generated over $j(F)$.
Then the following are true.
\begin{itemize}
\item[(1)] Suppose $j$ is a finite field extension.
Then $\sigma$ is separable if and only if $\tau$ is separable.
\item[(2)] Suppose $j(F)$ is algebraically closed in $E$ and $\tau$ is separable. Then $\sigma$ is separable.
\item[(3)] Suppose $\tau$ is separable. Then so is $\sigma$.
\end{itemize}
\end{lemma}

\begin{proof} (1) By multiplicativity
and since the extension $j(F)/j(\sigma(F))$ is equivalent to that of $F/\sigma(F)$, we have
$[E:j(\sigma(F))]_i=[E:j(F)]_i[F:\sigma(F)]_i$.
Similarly,
$[E:\tau(j(F))]_i=[E:\tau(E)]_i[E:j(F)]_i$.
Thus, by the equality $j(\sigma(F)) = \tau(j(F))$,
$[E:\tau(E)]_i=1$ holds if and only if so does $[F:\sigma(F)]_i=1$.

(2) Suppose $F/\sigma(F)$ is not separable. Then there exists $a\in F/\sigma(F)$ such that its minimal polynomial $f[x]\in \sigma(F)[x]$ is non-separable.
Since $j(F)$ is algebraically closed in $E$,
$j(f[x])\in \tau(E)[x]$ is also a non-separable minimal polynomial of $j(a)$, a contradiction.

(3) Let $j$ factor through $j_1:F\to K$ and $j_2: K\to E$ such that $j_1$ is algebraic and $j_2(K)$ is algebraically closed in $E$, which is the algebraic version of the geometric Stein factorization.
Since $E$ is finitely generated over $j(F)$, $j_1$ is a finite extension.
Note that there exists an extension $\varphi:K\to K$ such that $\varphi\circ j_1=j_1\circ \sigma$ and $\tau\circ j_2=j_2\circ \varphi$. Now (3) follows from (1) and (2).
\end{proof}

\begin{lemma}\label{lem-tame-des} Let $\pi:X\dashrightarrow Y$ be a dominant rational map of two normal varieties such that $k(Y)$ is separably closed in $k(X)$ (This holds when $\pi$ is a surjective projective morphism with connected fibres). Let $f:X\to X$ and $g:Y\to Y$ be two separable surjective endomorphisms such that $g\circ \pi=\pi\circ f$.
Then $\deg g^{\Gal} | \deg f^{\Gal}$.
Suppose further that $k(X)/k(Y)$ is a purely inseparable finite extension.
Then $\deg g^{\Gal} = \deg f^{\Gal}$.
\end{lemma}
\begin{proof}
Write $f:X_1=X\to X_2=X$ and $g:Y_1=Y\to Y_2\to Y$.
Write $f^{\Gal}:\overline{X}\to X_2$ and $g^{\Gal}:\overline{Y}\to Y_2$.
Then we have the following field extensions.
$$\xymatrix{
k(\overline{X}) & k(\overline{Y})\\
k(X_1)\ar[u] & k(Y_1)\ar[l]\ar[u]\\
k(X_2)\ar[u] & k(Y_2)\ar[l]\ar[u]
}$$

Write $k(Y_1)=k(Y_2)[a]$ for some $a\in k(Y_1)$ with $s[x]\in k(Y_2)[x] $ the minimal separable polynomial of $a$.
Since $k(Y)$ is separably closed in $k(X)$, $s[x]$ is also irreducible and separable in $k(X_2)$.
So $k(\overline{X})$ contains all the roots of $s[x]$.
Note that $k(\overline{Y})$ is the splitting field of $s[x]$.
Then $k(\overline{Y})$ is a subfield of $k(\overline{X})$.
Since $k(Y)$ is separably closed in $k(X)$, similarly we have $|k(\overline{Y})/k(Y_2)|=|k(\overline{Y})\cdot k(X_2)/k(X_2)|$.
In particular, $\deg g^{\Gal} | \deg f^{\Gal}$.

Suppose further $k(X)/k(Y)$ is purely inseparable.
Consider the separable degree $[k(Y_1)\cdot k(X_2):k(Y_2)]_s$ of $k(Y_1)\cdot k(X_2)/k(Y_2)$.
By multiplication law,
we have $[k(Y_1)\cdot k(X_2):k(Y_2)]_s=[k(Y_1)\cdot k(X_2):k(X_2)]_s=[k(Y_1):k(Y_2)]_s=|k(Y_1)/k(Y_2)|=|k(X_1)/k(X_2)|$.
In particular, $k(X_1)=k(Y_1)\cdot k(X_2)$ and hence $k(\overline{Y})\cdot k(X_2)$ contains $k(X_1)$.
Note that $k(\overline{Y})\cdot k(X_2)/ k(X_2)$ is Galois.
So $k(\overline{X})=k(\overline{Y})\cdot k(X_2)$.
Therefore, $\deg g^{\Gal} = \deg f^{\Gal}$.
\end{proof}

We will use the following fact in Galois theory.

\begin{lemma}\label{lem-gal-comp} Let $E_1/F$ and $E_2/F$ be the finite field extensions. Suppose $E_2/F$ is Galois. Then $K:=E_1\cdot E_2/E_1$ is Galois and $\frac{|E_2/F|}{|K/E_1|}$ is an integer.
\end{lemma}

\begin{lemma}\label{lem-gal-gal} Let $E_2/E_1$ and $E_3/E_2$ be two Galois extensions of degree co-prime to $p$. Let $F/E_1$ be the Galois closure of $E_3/E_1$. Then $p\nmid |F/E_1|$.
\end{lemma}
\begin{proof} Let $E_3=E_1[\alpha]=E_2[\alpha]$.
Then $F=E_3[\sigma_1(\alpha),\cdots, \sigma_n(\alpha)]$ where $\Gal(F/E_1)=\{\sigma_1,\cdots,\sigma_n\}$.
Let $F_i=E_3[\sigma_1(\alpha),\cdots, \sigma_i(\alpha)]$.
We show by induction on $i$ that $|F_i/E_1|$ is co-prime to $p$.
Assume that $F_i\neq F_{i+1}$.
Note that $\sigma_{i+1}(E_3)/\sigma_{i+1}(E_2)$ is Galois and $\sigma_{i+1}(E_2)=E_2$ since $E_2/E_1$ is Galois.
So $\sigma_{i+1}(E_3)/E_2$ is Galois with degree co-prime to $p$.
Note that $\sigma_{i+1}(E_3)=E_2[\sigma_{i+1}(\alpha)]$.
Then $F_{i+1}=F_i\cdot \sigma_{i+1}(E_3)$.
By Lemma \ref{lem-gal-comp}, $F_{i+1}/F_i$ is Galois with degree co-prime to $p$.
So the lemma is proved.
\end{proof}

\begin{lemma}\label{lem-gal-gen} Let $f_1:X_1\to X_2$ and $f_2:X_2\to X_3$ be two finite surjective morphisms of normal varieties.
Let $f:=f_2\circ f_1$.
Suppose $\deg f_1^{\Gal}$ and $\deg f_2^{\Gal}$ are both co-prime to $p$. Then $\deg f^{\Gal}$ is co-prime to $p$.
\end{lemma}
\begin{proof} Clearly, we may assume $f_1$ is Galois.
Let $K_i$ be the function field of $X_i$.
Then $K_1/K_2$ is Galois.
Let $E_3/K_3$ be the Galois closure of $K_2/K_3$ which has degree co-prime to $p$.
Since $K_1/K_2$ is Galois with degree co-prime to $p$, $E_3\cdot K_1/E_3$ is Galois with degree co-prime to $p$ by Lemma \ref{lem-gal-comp}.
Note that the Galois closure of $E_3\cdot K_1/K_3$ is the Galois closure of $K_1/K_3$.
So the lemma follows from Lemma \ref{lem-gal-gal}.
\end{proof}

\section{Proof of Theorem \ref{thm-pe}}

We first generalize \cite[Proposition 2.9]{MZ} to the following; see \cite[Definition 2.6]{MZ} for the notation and symbols involved below.

\begin{proposition}\label{prop-fx-x}
Let $\varphi:V\to V$ be an invertible linear map of a positive dimensional real normed vector space $V$.
Assume $\varphi^{\pm1}(C)=C$ for a convex cone $C\subseteq V$ such that $C$ spans $V$ and its closure $\overline{C}$ contains no line.
Let $q$ be a positive number. Then the conditions (i) and (ii) below are equivalent.
\begin{itemize}
\item[(i)] $\varphi(u)=q u$ for some $u\in C^\circ$ (the interior part of $C$).
\item[(ii)]
There exists a constant $N>0$, such that $\frac{||\varphi^i||}{q^i}< N$ for all $i\in \mathbb{Z}$.
\end{itemize}
Assume further the equivalent conditions (i) and (ii). Then the following are true.
\begin{itemize}
\item[(1)] $\varphi$ is a diagonalizable linear map with all eigenvalues of modulus $q$.
\item[(2)] Suppose $q>1$. Then, for any $v\in V$ such that $\varphi(v)-v\in C$, we have $v\in C$.
\end{itemize}
\end{proposition}

\begin{proof} Note that $\varphi^{\pm1}(C)=C$ implies $\varphi^{\pm1}(\overline{C})=\overline{C}$ and $\overline{C}^\circ=C^\circ$ since $C$ is a convex cone. So (i) and (ii) are equivalent and (1) is true by \cite[Proposition 2.9]{MZ}.

Assume the equivalent conditions (i) and (ii) and $q>1$.
Suppose $e:=\varphi(v)-v\in C$.
If $e=0$, then $v=0$ by (1) and since $q>1$.
Next we assume $e\neq 0$.

For $m\ge 1$, let $E_m$ be the convex cone generated by $\{\varphi^{-1}(e), \cdots, \varphi^{-m}(e)\}$.
Let $E_\infty$ be the convex cone generated by $\{\varphi^{-i}(e)\}_{i\ge 1}$.
Let $E$ be the convex cone generated by $\{\varphi^{-i}(e)\}_{i\in \mathbb{Z}}$.
Then all the above cones are subcones of $C$.
Note that $\varphi^{\pm}(E)=E$ and $\varphi^{-1}(E_\infty)\subseteq E_\infty$.
Let $W$ be the vector space spanned by $E$.
Since $e\neq 0$, $\dim(W)>0$.

We claim that $E_\infty$ spans $W$.
Let $W'$ be the vector space spanned by $E_\infty$.
Then $\varphi^{-1}(W')\subseteq W'$ and hence $\varphi^{\pm}(W')= W'$ since $W'$ is finite dimensional and $\varphi$ is invertible.
In particular, $\varphi^i(e)\in W'$ for any $i\in \mathbb{Z}$ and hence $W\subseteq W'$.
So the claim is proved.

Now we may assume $E_m$ spans $W$ for $m\gg 1$.
This implies $E_m^\circ\subseteq \overline{E}^\circ$.
Therefore, $s_m:=\sum\limits_{i=1}^m \varphi^{-i}(e)\in E_m^{\circ}\subseteq \overline{E}^{\circ}$.
Note that $\lim\limits_{n\to +\infty} \varphi^{-n}(v)= 0$ by (ii).
Then $v=\lim\limits_{n\to +\infty} v-\varphi^{-n}(v)=\lim\limits_{n\to +\infty}\sum\limits_{i=1}^n \varphi^{-i}(e)=s_m+\lim\limits_{n\to +\infty}\sum\limits_{i={m+1}}^n \varphi^{-i}(e)\in \overline{E}^\circ=E^{\circ}$.
In particular, $v\in E\subseteq C$.
\end{proof}

\begin{proof}[Proof of Theorem \ref{thm-pe}]
By the ramification divisor formula, $f^*(-K_X)-(-K_X)=R_f$.
Let $V$ be the real vector space $\N^{1}(X)$ with a fixed norm.
Let $C$ be the cone of classes of effective $\R$-Cartier divisors in $\N^1(X)$.
Then $C$ spans $V$ and its closure $\overline{C}=\PEC(X)$ contains no line.
Let $v=[-K_X]$ in $\N^1(X)$ and $\varphi=f^*|_{\N^1(X)}$.
Then $\varphi^{\pm}(C)=C$ and $\varphi(v)-v\in C$.
Since $f$ is polarized, $\varphi(u)=qu$ for some $u\in C^\circ$ and $q>1$.
Applying Proposition \ref{prop-fx-x} for the above settings, we have $v\in C$.
So $-K_X\equiv D$ for some effective $\Q$-Cartier divisor $D$.

Suppose $\Alb(X)=\Pic^0(\Pic^0(X)_{\red})$ is trivial.
Then so is $\Pic^0(X)_{\red}$. Hence $-K_X\sim_{\Q} D$ and $\kappa(X,-K_X)\ge 0$.
\end{proof}

\begin{remark} In Theorem \ref{thm-pe}, one can remove the assumption ``$\Q$-Gorenstein"  and still assert that $-K_X$ is weakly numerically equivalent to some effective Weil $\Q$-divisor (cf.~\cite[Definition 2.2]{MZ}).
Indeed, the same proof above works by applying Proposition \ref{prop-fx-x} to the cone generated by the classes of effective Weil $\R$-divisors in $\N_{n-1}(X)$ (cf.~\cite[Definitions 2.2 and 2.4]{MZ}).
\end{remark}

\section{$K_X$ pseudo-effective case}

We first recall a useful result of \cite[Lemma 2.1]{Na-Zh} or \cite[Proposition 2.9]{MZ}.
\begin{lemma}\label{lem-eigen} Let $f:X\to X$ be a numerically $q$-polarized endomorphism of a normal projective variety $X$.
Suppose $f^*D\equiv aD$ (numerical equivalence) for some $\R$-Cartier divisor $D$ and $a\in \R$.
Then either $|a|=q$ or $D\equiv 0$.
\end{lemma}

The following remark on Lemma \ref{lem-eigen} is an answer to a question by the referee.

\begin{remark}
It can happen that $D \not\equiv 0$ and $a = -q$ in Lemma \ref{lem-eigen}. For example, let $X:=\mathbb{P}^1\times \mathbb{P}^1$ and let $f:X\to X$ via $f([a:b],[c:d])=([c^2:d^2],[a^2:b^2])$.
Denote by $p_1: X\to \mathbb{P}^1$ and $p_2: X\to \mathbb{P}^1$ the two natural projections. 
Let $H_1:=p_1^*H$ and $H_2:=p_2^*H$ where $H:=[1:0]\in \mathbb{P}^1$.
Then $\text{Pic}(X)=\mathbb{Z}H_1\oplus \mathbb{Z}H_2$, $H_1+H_2$ is ample, and $f^*(H_1+H_2)=2H_2+2H_1$.
So $f$ is $2$-polarized.
Note that $f^*(H_1-H_2)=2H_2-2H_1=-2(H_1-H_2)$ and $H_1-H_2\not\equiv 0$.
So it is possible that the eigenvalue is negative.
\end{remark}

\begin{lemma}\label{lem-pe-qe} Let $f:X\to X$ be a numerically $q$-polarized separable endomorphism of an $n$-dimensional $\Q$-Gorenstein normal projective variety $X$. Then the following are equivalent.
\begin{itemize}
\item[(1)] $f$ is quasi-\'etale.
\item[(2)] $K_X \equiv f^*K_X$.
\item[(3)] $K_X\equiv 0$.
\item[(4)] $K_X$ is pseudo-effective.
\end{itemize}
\end{lemma}
\begin{proof}
For (2) $\Rightarrow$ (3), we refer to Lemma \ref{lem-eigen}.

We then only need to show (4) $\Rightarrow$(1).
Suppose that $f^\ast H \equiv qH$ for some ample Cartier divisor $H$ of $X$ and $q>1$.
Using $(f^{\ast}H)^{n-1}=f^{\ast}H\cdots f^{\ast}H$ to intersect both sides of the ramification divisor formula $K_X=f^{\ast}K_X+R_f$, we obtain
$$(q-1)K_X\cdot H^{n-1}+R_f\cdot H^{n-1}=0.$$
Since $K_X$ is pseudo-effective and $R_f$ is effective, $K_X\cdot H^{n-1}=R_f\cdot H^{n-1}=0$ and hence  $R_f=0$. So (1) is proved.
\end{proof}

\begin{remark} Alternatively, by Theorem \ref{thm-pe}, $-K_X$ is pseudo-effective and hence Lemma \ref{lem-pe-qe} is straightforward.
\end{remark}


\begin{lemma}\label{app-lem-k-eff} Let $f:X\to X$ be a numerically $q$-polarized separable endomorphism of a normal projective variety $X$. Suppose $mK_X$ is effective for some integer $m>0$.
Then $X$ is $\Q$-Gorenstein and $f$ is quasi-\'etale; precisely, $mK_X\sim 0$.
\end{lemma}
\begin{proof} Suppose $f^*H\equiv qH$ for some ample Cartier divisor $H$.
Using $(f^{\ast}H)^{n-1}=f^{\ast}H\cdots f^{\ast}H$ to intersect both sides of the ramification divisor formula $K_X=f^{\ast}K_X+R_f$, we obtain
$$(q-1)K_X\cdot H^{n-1}+R_f\cdot H^{n-1}=0.$$
Then $K_X\cdot H^{n-1}=0$ and $R_f=0$.
Since $mK_X$ is effective, $mK_X\sim 0$.
\end{proof}

We generalize \cite[Theorem 3.3]{Na-Zh} to the case of positive characteristic as follows:

\begin{theorem}\label{thm-FQ}
Let $X$ be a strongly F-regular $\Q$-Gorenstein normal projective variety over the field $k$ of characteristic $p>0$.
Let $f:X\to X$ be a polarized separable endomorphism. Suppose $\dim(X)>1$ and $K_X$ is pseudo-effective.
Then there is a quasi-\'etale cover $\pi:V\to X$ such that:
\begin{itemize}
\item[(1)] $V$ is smooth; and
\item[(2)] $c_1(V)=0$ and $c_2(V)\cdot A_V^{n-2}=0$ for some ample Cartier divisor $A_V$ of $V$.
\end{itemize}
In particular,  if $X$ is a surface and $p>3$, then $V$ is $Q$-abelian and hence $X$ is $Q$-abelian.
\end{theorem}

\begin{proof} By Lemma \ref{lem-pe-qe}, $f$ is quasi-\'etale and $K_X\equiv 0$. By \cite[Main Theorem]{BCGST}, the claim in the proof of \cite[Theorem 3.3]{Na-Zh}
holds true for any characteristic. Therefore the same argument in the proof of \cite[Theorem 3.3]{Na-Zh} works to conclude (1) and (2).

If $X$ is a surface and $p>3$, (1) and (2) imply that $V$ is either an abelian surface or a hyperelliptic surface by the classification of surfaces (cf.~\cite{BM}).
In particular, $V$ is $Q$-abelian and so is $X$.
\end{proof}

\begin{remark} In the case of characteristic 0, Theorem \ref{thm-FQ} (1) and (2) imply $V$ is $Q$-abelian (cf.~\cite{Be} and \cite{Ya}).
\end{remark}

%
%
%
%
\section{Proof of Theorems \ref{ap-cor-dual} and \ref{thm-torsion}}\label{sec-app}

Let $f:X\to X$ be a surjective endomorphism of a normal projective variety $X$ over $k$.
Denote by
$$\Alb(X):=\Pic^0(\Pic^0(X)_{\red})$$
which is an abelian variety.  Then there is a canonical morphism $\alb_X:X\to \Alb(X)$ such that:
$\alb_X(X)$ generates $\Alb(X)$ and for every morphism $\varphi:X\to A$ from $X$ to an abelian variety $A$,
there exists a unique morphism $\psi:\Alb(X)\to A$ such that $\varphi = \psi\circ \alb_X$
(cf. \cite[Remark 9.5.25]{FGI}).

In the birational category, there exists an abelian variety $\mathfrak{Alb}(X)$ together with a rational map $\mathfrak{alb}_X:X\dashrightarrow \mathfrak{Alb}(X)$ such that: $\mathfrak{alb}_X(X)$ generates $\mathfrak{Alb}(X)$ and for every rational map $\varphi:X\dashrightarrow A$ from $X$ to an abelian variety $A$,
there exists a unique morphism $\psi:\mathfrak{Alb}(X)\to A$ such that $\varphi = \psi\circ \mathfrak{alb}_X$ (cf.~\cite[Chapter II.3]{Lang}).
If $\mathfrak{alb}_X$ is a morphism, then $\mathfrak{alb}_X$ and $\alb_X$ are the same.

By the above two universal properties, $f$ descends to surjective endomorphisms on $\Alb(X)$ and $\mathfrak{Alb}(X)$.

The purpose of this section is to prove Theorems \ref{ap-cor-dual}, \ref{thm-torsion} and the following:

\begin{theorem}\label{thm-num-lin}
Let $f:X\to X$ be a numerically polarized separable endomorphism of a normal projective variety
over the field $k$. Then $f$ is polarized.
\end{theorem}

%

\begin{lemma}\label{app-lem-stein} Let $\pi:X\to Y$ be a morphism between normal projective varieties.
Let $p_1:X\to Z$ and $p_2:Z\to Y$ be the Stein factorization of $\pi$.
Suppose there are two finite endomorphisms $f:X\to X$ and $g:Y\to Y$ such that $g\circ \pi=\pi\circ f$.
Then there is a finite endomorphism $h:Z\to Z$ such that $h\circ p_1=p_1\circ f$ and $g\circ p_2=p_2\circ h$.
\end{lemma}
\begin{proof} Let $X\xrightarrow{q_1}Z'\xrightarrow{q_2}Z$ be the Stein factorization of $p_1\circ f$. Then $X\xrightarrow{q_1}Z'\xrightarrow{p_2\circ q_2}Y$ and $X\xrightarrow{p_1}Z\xrightarrow{g\circ p_2}Y$ are both the Stein factorizations of $g\circ \pi=\pi\circ f$. So $q_1=\sigma\circ p_1$ with $\sigma:Z\to Z'$ being an isomorphism. Therefore, $h=q_2\circ\sigma$ is a finite endomorphism of $Z$ as required.
\end{proof}

\begin{lemma}\label{app-lem-sep-e}Let $\pi:X\to Y$ be a finite surjective morphism of two normal varieties. Let $f:X\to X$ and $g:Y\to Y$ be two surjective endomorphisms such that $g\circ \pi=\pi\circ f$.
Let $Z$ be a normal variety such that $\pi=\tau\circ\sigma$ where $\sigma:X\to Z$ is purely inseparable and $\tau:Z\to Y$ is separable.
Then there is a surjective endomorphism $h:Z\to Z$ such that $h\circ \sigma=\sigma\circ f$ and $g\circ \tau=\tau\circ h$.
\end{lemma}
\begin{proof}
Let $V_1:=\SSpec A_1$ be an affine open dense subset of $Y$.
Denote by $U_1:=\pi^{-1}(V_1):=\SSpec B_1$ and $W_1:=\tau^{-1}(V_1):=\SSpec C_1$.
Let $V_2:=g^{-1}(V_1):=\SSpec A_2$.
Denote by $U_2:=\pi^{-1}(V_2):=\SSpec B_2$ and $W_2:=\tau^{-1}(V_2):=\SSpec C_2$.
For convenience, we denote by $g:A_1\to A_2$, $f:B_1\to B_2$ the induced ring injections.
Then we have the following commutative diagram.
$$\xymatrix{
  A_1 \ar[r]^{g}\ar[d] & A_2\ar[d]        \\
  C_1 \ar[d] & C_2\ar[d]             \\
  B_1 \ar[r]^{f}& B_2}$$
Note that every element of $C_1$ is separable over $Q(A_1)$ (the quotient field of $A_1$).
So every element of $f(C_1)$ is separable over $g(Q(A_1))$ and hence separable over $Q(A_2)$.
Therefore, $f(C_1)\subseteq Q(C_2)$.
Note that $C_1$ is the integral closure of $A_1$ in $Q(C_1)$.
So $f(C_1)$ is integral over $g(A_1)$ and hence integral over $A_2$.
In particular, $f(C_1)\subseteq C_2$.
Then we have a ring homomorphism $h:=f|_{C_1}:C_1\to C_2$ such that the above diagram commutes with $h$.
\end{proof}

\begin{lemma}\label{ap-lem-gen} Let $f:A\to A$ be a surjective endomorphism of an abelian variety $A$. Suppose there is a closed subvariety $Z$ of $A$ such that
\begin{itemize}
\item[(1)] $f(Z)=Z$,
\item[(2)] $f|_Z$ is numerically polarized separable, and
\item[(3)] $Z$ generates $A$.
\end{itemize}
Then $Z=A$.
\end{lemma}

\begin{proof} We may assume $\dim(A)>0$.
Let $a=f(0)$. Then $f=T_a \circ f_0$ where $f_0$ is an isogeny,
and $T_a$ is the translation.
Let $\Stab_A(Z)$ be the maximal closed subgroup of $A$ such that $\Stab_A(Z)+Z=Z$.
If $\dim(\Stab_A(Z))=0$, then $f|_Z$ is bijective by \cite[Proposition 5.3]{PR}, a contradiction to the condition (2).
Suppose $\dim(\Stab_A(Z))>0$. Let $B$ be the neutral component of $\Stab_A(Z)$. Then $f_0(B)=B$.
If $B=A$, then $Z=A$, and we are done.

Suppose $B\subsetneq A$.
Let $p:A\to A/B$ be the natural projection. Then there is a surjective endomorphism $\bar{f}:A/B\to A/B$ such that $\bar{f}\circ p=p\circ f$.
Since $Z$ generates $A$, $\dim(p(Z))>0$.
Note that $\bar{f}|_{p(Z)}$ is numerically polarized separable (cf.~Lemma \ref{sepdeg} and \cite[Theorem 3.11]{MZ}) and $\dim(\Stab_{A/B}(p(Z)))=0$. We get a contradiction again by \cite[Proposition 5.3]{PR}.
\end{proof}

\begin{lemma}\label{ap-lem-surj} Let $f:X\to X$  be a numerically $q$-polarized separable endomorphism of a normal projective variety $X$.
Then $\alb_X:X\to \Alb(X)$ is surjective and the induced endomorphism
$g:\Alb(X)\to \Alb(X)$ is again numerically $q$-polarized separable.
\end{lemma}

\begin{proof}
Let $Z := \alb_X(X)$ which generates $\Alb(X)$.
Then $g(Z) = Z$ and $g|_Z$ is numerically $q$-polarized separable by \cite[Theorem 3.11]{MZ} and Lemma \ref{sepdeg}.
Now Lemma \ref{ap-lem-gen} implies that $Z = \Alb(X)$.
\end{proof}

\begin{lemma}\label{app-lem-finite-abe} Let $f:X\to X$ be a numerically polarized separable endomorphism of a normal projective variety $X$.
Suppose $\alb_X$ is finite and surjective.
Then $\alb_X$ is an isomorphism and $X$ is an abelian variety.
\end{lemma}
\begin{proof}
Let $Z$ be a normal projective variety such that $\alb_X=\tau\circ\sigma$ where $\sigma:X\to Z$ is purely inseparable and $\tau:Z\to \Alb(X)$ is separable.
By Lemma \ref{app-lem-sep-e}, there is a surjective endomorphism $h:Z\to Z$ such that $h\circ \sigma=\sigma\circ f$ and $g\circ \tau=\tau\circ h$.
Note that $h$ is numerically polarized separable by \cite[Theorem 3.11]{MZ} and Lemma \ref{sepdeg}.
Since $\tau$ is separable, $K_Z$ is effective by the ramification divisor formula.
So $K_Z\sim 0$ by Lemma \ref{app-lem-k-eff} and hence $\tau$ is \'etale by the purity of branch loci.
In particular, $Z$ is an abelian variety (cf.~\cite[Chapter IV, 18]{Mu}).
By the universal property of $\alb_X$, $\tau$ is an isomorphism.
So we may assume $\alb_X$ is purely inseparable.
By Lemma \ref{app-lem-k-eff}, either $mK_X\sim 0$ for some $m>0$ or $mK_X$ is not effective for any $m>0$.
Then $\alb_X$ is separable by \cite[Proposition 1.4]{HPZ}.
So $\alb_X$ is an isomorphism and hence $X$ is an abelian variety.
\end{proof}

Let $A$ be an abelian variety.
We recall some facts from \cite{Mu}.
Let $0 \ne n\in \mathbb{Z}$. Denote by $n_A:A\to A$ the isogeny sending $a$ to $na$.
Let $L$ be a Cartier divisor of $A$.
Then $n_A^*L\equiv n^2L$.
Let
$$\begin{aligned}
\phi_L: \, A \, \to \, \widehat{A}:=\Pic^0(A) \\
a \, \mapsto \, T_a^*L - L
\end{aligned}$$
where $T_a$ is the translation map by $a$.
Note that $L\in \Pic^0(A)$ if and only if $\phi_L$ is a trivial map.
If $L$ is ample, then $\phi_L$ is an isogeny.

Let $f:X\to Y$ be a morphism of normal projective varieties. Denote by $\widehat{f}:=f^*|_{\Pic^0(Y)_{\red}}:\Pic^0(Y)_{\red}\to \Pic^0(X)_{\red}$ the dual of $f$.
Let $g:\Alb(X)\to \Alb(Y)$ be the induced morphism,
where $\Alb(X)=\Pic^0(\Pic^0(X)_{\red})$.
Then $\widehat{g}=\widehat{f}$.

\begin{lemma}\label{ap-lem-dual} Let $f:A\to A$ be a numerically $q$-polarized endomorphism of an abelian variety $A$.
Then the dual $\widehat{f}:\Pic^0(A)\to \Pic^0(A)$ is numerically $q$-polarized.
\end{lemma}

\begin{proof} Write $f=T_a \circ f_0$ where $f_0:A\to A$ is an isogeny and $T_a$ is the translation.
Note that $f_0$ is also numerically $q$-polarized and $\widehat{f}=\widehat{f_0}$. So we may assume $f$ is an isogeny.
By the assumption, $f^*L-qL\in \Pic^0(A)$ for some ample line bundle $L$.
Note that $q_{\widehat{A}}\circ \phi_L=\phi_{f^*L}=\widehat{f}\circ\phi_L\circ f$, i.e.
the following is a commutative diagram
$$\xymatrix{
A\ar[r]^{f}\ar[d]^{\phi_L}&A\ar[dd]^{\phi_L}\\
\widehat{A}\ar[d]^{q_{\widehat{A}}}\\
\widehat{A}&\widehat{A}\ar[l]^{\widehat{f}} . \\
}$$

Taking the pullback of the above, we have:
$$\xymatrix{
\N^1(A)&\N^1(A)\ar[l]_{f^*}                      \\
\N^1(\widehat{A})\ar[u]^{\phi_L^*}\\
\N^1(\widehat{A})\ar[r]^{\widehat{f}^*}\ar[u]^{q_{\widehat{A}}^*}&\N^1(\widehat{A})\ar[uu]^{\phi_L^*}\ar[lu]_{\varphi}\\
}$$
where the isomorphism
$\varphi=q_{\widehat{A}}^*\circ(\widehat{f}^*)^{-1}$ makes the whole diagram commutative.
We now apply Proposition \ref{prop-fx-x} (with $C$ being the cone of nef $\R$-Cartier divisors there)  and check the norm condition (2) in it.
By the assumption, $f^*$ satisfies the norm condition, so does $\varphi$ (or equivalently $\varphi^{-1}$),
since $\phi_L^*$ is an isomorphism.
Note that the scalar map $q_{\widehat{A}}^*=q^2\id_{\N^1(\widehat{A})}$
commutes with $\varphi$.
Then our $\widehat{f}^* = \varphi^{-1} \circ q_{\widehat{A}}^*$
satisfies the norm condition.
Hence $\widehat{f}$ is numerically $q$-polarized.
\end{proof}

\begin{lemma}\label{ap-lem-X-dual} Let $f:X\to X$ be a numerically $q$-polarized separable endomorphism of a normal projective variety $X$. Then the dual $\widehat{f}:\Pic^0(X)_{\red}\to \Pic^0(X)_{\red}$ is numerically $q$-polarized.
\end{lemma}

\begin{proof} By Lemma \ref{ap-lem-surj}, the induced endomorphism $g:\Alb(X)\to \Alb(X)$ is numerically $q$-polarized separable.
Note that $\widehat{f}=\widehat{g}$.
Now the lemma follows from Lemma \ref{ap-lem-dual}.
\end{proof}

\begin{lemma}\label{ap-lem-2} Let $f:A\to A$ be a numerically $q$-polarized isogeny of an abelian variety $A$ with $n=\dim(A)>0$. Let $\lambda$ be an integer. If $f-\lambda_A$ is not surjective, then $q=\lambda^2$.
\end{lemma}
\begin{proof} Let $B$ the kernel of $f-\lambda_A$. Then $B$ is a closed subgroup of $A$. Let $B_0$ be the neutral component of $B$. Then we have the following commutative diagram
$$\xymatrix{
B_0\ar[r]^{\lambda_{B_0}}\ar@{^{(}->}[d]^{i}&B_0\ar@{^{(}->}[d]^{i}                       \\
A\ar[r]^{f}&A\\
}$$
where $i$ is the inclusion map.

Like $f$, the $\lambda_{B_0}$ is also $q$-polarized. Hence $q=\lambda^2$ if $\dim(B_0)>0$.
\end{proof}

\begin{proof}[Proof of Theorem \ref{thm-num-lin}]
We may assume $\dim(X)>0$.
Suppose $f^*H\equiv qH$ for some ample Cartier divisor $H$ and $q>0$.
Let $D=f^*H-qH$. Replacing $H$ by a multiple, we may assume $D\in \Pic^0(X)_{\red}$.
Note that $\widehat{f}$ is numerically $q$-polarized by Lemma \ref{ap-lem-X-dual}.
So $\widehat{f}-q_{\Pic^0(X)_{\red}}$ is a surjective endomorphism of $\Pic^0(X)_{\red}$ by Lemma \ref{ap-lem-2}.
In particular, there is a Cartier divisor $L\in \Pic^0(X)_{\red}$ such that $f^*L-qL\sim D$.
Then $f^*H'\sim qH'$ for $H'=H-L$. Note that $H'\equiv H$ is ample. So $f$ is polarized.
\end{proof}

We will use Lemmas \ref{prop-irr} and \ref{fibres-rc+irr} proved in Section \ref{sec-qab}.

\begin{proof} [Proof of Theorem \ref{ap-cor-dual}]
By Lemma \ref{ap-lem-surj} and Theorem \ref{thm-num-lin}, $\alb_X$ is surjective and $g$ is $q$-polarized separable.
Taking the Stein factorization of $\alb_X$, we have $\varphi:X\to Y$ and $\psi:Y\to \Alb(X)$ such that $\varphi_\ast\mathcal{O}_X=\mathcal{O}_Y$ and $\psi$ is a finite morphism.
Then $f$ descends to a polarized separable endomorphism $f_Y:Y\to Y$ by Lemma \ref{app-lem-stein}, \cite[Theorem 3.11]{MZ}, Theorem \ref{thm-num-lin} and Lemma \ref{sepdeg}.
By the universal property, $\psi=\alb_Y$.
So $\psi$ is an isomorphism by Lemma \ref{app-lem-finite-abe}, and we can identify
$\alb_X : X \to \Alb(X)$ with $\varphi: X \to Y$.
By Lemmas \ref{prop-irr} and \ref{fibres-rc+irr}, all the fibres of $\alb_X$ are irreducible and equi-dimensional.
So (1) is proved.

For (2), let $W$ be the normalization of the graph of $\mathfrak{alb}_X$.
Then $\Alb(W)=\mathfrak{Alb}(W)=\mathfrak{Alb}(X)$.
Note that $f$ lifts to a numerically $q$-polarized separable endomorphism $f_W:W\to W$ (cf.~Theorem \ref{polar_des}).
Therefore by (1), the induced endomorphism $h:\mathfrak{Alb}(X)\to \mathfrak{Alb}(X)$ is $q$-polarized separable.
Since $\alb_W$ is surjective by (1), $\mathfrak{alb}_X$ is dominant.

(3) follows from Lemma \ref{ap-lem-X-dual}.
\end{proof}

\begin{proof}[Proof of Theorem \ref{thm-torsion}] We may assume $\dim(X)>0$.
By Lemma \ref{lem-pe-qe}, $K_X \equiv 0$, $f$ is quasi-\'etale, and hence $K_X = f^*K_X$.
Further, $mK_X\in \Pic^0(X)_{\red}$ for some $m>0$.
By Theorem \ref{ap-cor-dual}, $\widehat{f}$ is numerically $q$-polarized.
Hence, $\widehat{f}-1_{\Pic^0(X)_{\red}}$ is surjective by Lemma \ref{ap-lem-2} and its kernel is finite.
Therefore, $mK_X\in \Ker(\widehat{f}-1_{\Pic^0(X)_{\red}})$ has finite order in $\Pic^0(X)_{\red}$.
\end{proof}

\section{Descending of polarized endomorphisms}

\begin{remark}\label{rmk-equ1}
Let $f:X\to X$ be a numerically $q$-polarized endomorphism of a normal projective variety $X$ over the field $k$ of characteristic $p\ge 0$.
We recall the key results of \cite[Section 6, Lemmas 6.1 - 6.6]{MZ}. When $p=0$, we have shown that the divisorial contractions, the Fano contractions, the flipping contractions and the flips of extremal rays (if exist) are all $f$-equivariant (after replacing $f$ by a positive power).
When $p>0$, \cite[Lemma 6.2, Remark 6.3 and Lemma 6.6]{MZ} still hold with the same proof.
\cite[Lemmas 6.4 - 6.5]{MZ} are based on  \cite[Lemma 6.1]{MZ}. Apparently, \cite[Lemma 6.1]{MZ} does not hold when the $f$ there is the geometric Frobenius endomorphism which is a bijection set-theoretically;
so we need to restrict to those $f : X \to X$
which are numerically $q$-polarized with $p$ and $q$ co-prime;
such $f$ is separable since $\deg f = q^{\dim(X)}$ and hence $(p,\deg f)=1$.
We remark that we cannot simply weaken the assumption to $f$ being separable
because such separable $f$ may not restrict to a separable morphism on a subvariety;
see Remark \ref{rmk-sep} for an example where \cite[Lemma 6.1]{MZ} fails to hold
yet $f$ is separable (and $p$ divides $\deg f$).
In summary, we prove Lemma \ref{finite-orbit} below, replacing \cite[Lemma 6.1]{MZ}.
\end{remark}

The proof of Lemma \ref{finite-orbit} below is almost identical to that of
\cite[Lemma 6.1]{MZ}, but for convenience, we reproduce here and highlight the argument in the proof
where we need the assumption of $f$ being numerically $q$-polarized with $p$ and $q$ co-prime.

\begin{lemma}\cite[Lemma 6.1]{MZ}\label{finite-orbit} Let $f:X\to X$  be a numerically $q$-polarized separable endomorphism of a projective variety $X$ over the field $k$ of characteristic $p$.
Assume $A\subseteq X$ is a closed subvariety with $f^{-i}f^i(A) = A$ for
all $i\ge 0$.
Assume further either one of the following conditions.
\begin{itemize}
\item[(1)] $A$ is a prime divisor of $X$.
\item[(2)] $p$ and $q$ are co-prime.
\end{itemize}
Then $M(A) := \{f^i(A)\,|\, i \in \Z\}$ is a finite set.
\end{lemma}
\begin{proof}
Suppose that $\dim(X)=n$, and $f^{\ast}H\equiv qH$ for some ample Cartier divisor $H$
and integer $q>1$.
Let $M_{\ge0}(A):=\{f^i(A)\,|\, i \ge 0\}$.

If $M_{\ge0}(A)$ is a finite set, $f^{r_1}(A)=f^{r_2}(A)$ for some $0<r_1<r_2$.
Note that $f^{-i}f^i(A) = A$ for all $i \ge 0$.
Then for any $i\ge 0$, $f^{-i}(A)=f^{-i}f^{-sr_1}f^{sr_1}(A)=f^{-i}f^{-sr_1}f^{sr_2}(A)=f^{sr_2-sr_1-i}(A)\in M_{\ge 0}(A)$ if $s\gg 1$.
So it suffices to show $M_{\ge0}(A)$ is a finite set. It is trivial if $A=X$.

Set $k := \dim(A)<n=\dim(X)$, $A_i := f^i(A) \, (i\ge 0)$.
Let $\Sigma$ be the union of $\Sing (X)$, $f^{-1}(\Sing (X))$, and the irreducible components in the ramification divisor $R_f$ of $f$.

We first claim that $A_i$ is contained in $\Sigma$ for infinitely many $i$.
Otherwise, replacing $A$ by some $A_{i_0}$, we may assume that $A_i$ is not contained in $\Sigma$ for all $i\ge 0$.
So we have $f^{\ast}A_{i+1}= a_iA_i$ with $a_i\in \mathbb{Z}_{>0}$ and
$$q^nH^{k}\cdot A_{i+1}=(f^{\ast}H)^{k}\cdot f^{\ast}A_{i+1}=a_iq^{k}H^{k}\cdot A_{i},$$
$$1\leq H^{k}\cdot A_{i+1}=\frac{a_i}{q^{n-k}}\cdots \frac{a_1}{q^{n-k}}H^{k}\cdot A_1.$$
Thus for infinitely many $i$, $a_i\geq q^{n-k}$ and $A_i\subseteq \Sigma$, a contradiction. This proves the claim.

If we assume the condition (1), by the claim, then $f^{r_1}(A)=f^{r_2}(A)$ for some $0<r_1<r_2$.
So $|M_{\ge 0}(A)|<r_2$.

Next we assume the condition (2), $k\leq n-2$ and $|M_{\ge 0}(A)|=\infty$.
Let $B$ be the Zariski-closure of the union of those $A_{i_1}$ contained in $\Sigma$.
Then $k+1\leq \dim(B)\le n-1$, and $f^{-i}f^i(B) = B$ for all $i \ge 0$. Choose $r \ge 1$ such that $B' := f^r(B), f(B'), f^2(B'), \cdots$ all have the same number of irreducible components.
Let $X_1$ be an irreducible component of $B'$ of maximal dimension.
Then $k+1\leq \dim(X_1)\le n-1$ and $f^{-i}f^i(X_1) = X_1$ for all $i \ge 0$. By induction, $M_{\ge 0}(X_1)$ is a finite set.
So we may assume that $f^{-1}(X_1)=X_1$, after replacing $f$ by a positive power and $X_1$ by its image.
\textbf{Note that $f|_{X_1}$ is numerically $q$-polarized.}
Now the codimension of $A_{i_1}$ in $X_1$ is smaller than that of $A$ in $X$.
By induction, $M_{\ge 0}(A_{i_1})$ and hence $M_{\ge 0}(A)$ are finite.
This is a contradiction.

So under either the condition (1) or (2), $M_{\ge 0}(A)$ is always a finite set.
Therefore the lemma is proved.
\end{proof}

\begin{remark}\label{rmk-sep} If we simply assume $f$ is numerically polarized separable in Lemma \ref{finite-orbit}, then we cannot say $f|_{X_1}$ is again separable during the induction of the proof. Indeed, we have the following example.
Let $X:=\mathbb{P}^3$, $p=3$, and $f([a:b:c:d])=[a^3+acd:b^3+bcd:c^3+c^2d:d^3-cd^2]$.
Then $f$ is $3$-polarized separable.
Let $X_1:=\{c=0, d=0\}\cong \mathbb{P}^1$.
Then $f^{-1}(X_1)=X_1$ and $f|_{X_1}([a:b])=[a^3:b^3]$ which is a geometric Frobenius of $\mathbb{P}^1$.
Clearly, $f|_{X_1}$ is not separable.
Since $f|_{X_1}$ is bijective, $f^{-i}f^i(A) = A$ for any closed point $A$ of $X_1$.
However, $M(A)$ is infinite when $A= \{[1:b:0:0]\}$ and $b$ is not a root of the unity.
\end{remark}

\begin{remark}\label{rmk-div} By Lemma \ref{finite-orbit} and \cite[Lemma 6.4]{MZ}, a divisorial contraction is $f$-equivariant (after replacing $f$ by a positive power) for  numerically polarized separable $f$.
\end{remark}

Theorem \ref{polar_des} below in positive characteristic is obtained with the same proof as in \cite{MZ}.

\begin{theorem}\label{polar_des}\cite[Theorem 3.11 and Corollary 3.12]{MZ} Let $\pi:X\dasharrow Y$ be a dominant rational map between two normal projective varieties and let $f : X \to X$ and $g : Y \to Y$ be two surjective endomorphisms such that $g\circ \pi=\pi\circ f$. Then the following are true.
\begin{itemize}
\item[(1)] Suppose $f$ is numerically $q$-polarized. Then so is $g$.
\item[(2)] Suppose $\pi$ is generically finite. Then $f$ is numerically $q$-polarized if and only if so is $g$.
\end{itemize}
\end{theorem}

\begin{remark}\label{rmk-num-lin} Let $f:X\to X$ be a separable surjective endomorphism of a normal projective variety $X$.
Then we may replace ``numerically $q$-polarized by ``$q$-polarized'' in Theorem \ref{polar_des}.
In the case of characteristic $0$, Nakayama and Zhang showed the equivalence of $f$ being numerically $q$-polarized and $f$ being $q$-polarized (cf.~\cite[Lemma 2.3]{Na-Zh}; indeed their proof works for any projective $X$). For arbitrary characteristic, the same equivalence holds, with a new proof given
in Theorem \ref{thm-num-lin}.
\end{remark}

\begin{remark}\label{rmk-emmp} We here summarize the version of the polarized equivariant MMP in
positive characteristic by using Remarks \ref{rmk-equ1} and \ref{rmk-div}, Theorem \ref{polar_des} and Remark \ref{rmk-num-lin}.
Let $X$ be a normal projective variety over the field $k$ of characteristic $p>0$. Let $f:X\to X$ be a $q$-polarized separable endomorphism. Suppose there exists a finite sequence of MMP:
$$X=X_1\dashrightarrow \cdots \dashrightarrow X_i \dashrightarrow \cdots \dashrightarrow X_r,$$ with every $X_i \dashrightarrow X_{i+1}$ a divisorial contraction, a flip or a Fano contraction, of a $K_{X_i}$-negative extremal ray.
Assume that $(p,q)=1$.
Then after replacing $f$ by a positive power, this sequence is $f$-equivariant, and $f=f_1$ descends to $q$-polarized $f_i$ on each $X_i$.

Note that if there is no flip involved in the sequence, then without the assumption $(p,q)=1$, this sequence is still $f$-equivariant (after replacing $f$ by a positive power), and $f=f_1$ descends to $q$-polarized separable $f_i$ on each $X_i$; see \cite[Remark 6.3]{MZ} for the Fano contraction, Remark \ref{rmk-div} for the divisorial contraction and  Lemma \ref{sepdeg} for the separable property.
\end{remark}

\section{Endomorphisms compatible with a fibration}
We recall the following useful result from \cite[Theorem, Page 220]{Dan}.
Note that the assumption of normality below is necessary.
\begin{lemma}\label{om-lem}
Let $f:X\to Y$ be a finite surjective morphism of two varieties
with $Y$ being normal.
Then $f$ is an open map.
\end{lemma}
%

\begin{lemma}\label{open-lem} Let $f:X\to Y$ be a finite surjective morphism of two varieties with $Y$ being normal.
Let $S$ be a subset of $Y$.
Then $f^{-1}(\overline{S})=\overline{f^{-1}(S)}$.
\end{lemma}

\begin{proof} Clearly, $\overline{f^{-1}(S)}\subseteq f^{-1}(\overline{S})$. Suppose $x\in f^{-1}(\overline{S})-\overline{f^{-1}(S)}$.
Then there is an open neighborhood $x\in U$ with $U\cap \overline{f^{-1}(S)}=\emptyset$. By Lemma \ref{om-lem}, $f(U)$ is open. Since $f(U)\cap S=\emptyset$, $f(x)\in f(U)\cap \overline{S}=\emptyset$, a contradiction.
\end{proof}

\begin{lemma}\label{surj-lem} Let $\pi:X\to Y$ be a surjective morphism of two normal varieties with connected fibres. Let $f:X\to X$ and $g:Y\to Y$ be two surjective endomorphisms such that $g\circ \pi=\pi\circ f$. Then $f|_{\pi^{-1}(Z)}:\pi^{-1}(Z)\to \pi^{-1}(g(Z))$ is surjective for any subset $Z$ of closed points in $Y$.
\end{lemma}
\begin{proof} Clearly, it suffices to consider the case when $Z$ is a closed point of $Y$.
Suppose there is a closed point $y$ of $Y$ such that $f|_{\pi^{-1}(y)}:\pi^{-1}(y)\to \pi^{-1}(g(y))$ is not surjective.
Let $S=g^{-1}(g(y))-\{y\}$.
Then $S\neq \emptyset$ and $U:=X-\pi^{-1}(S)$ is a proper open dense subset of $X$.
By Lemma \ref{om-lem}, $f(U)$ is an open dense subset of $X$. In particular, $f(U)\cap \pi^{-1}(g(y))$ is open in $\pi^{-1}(g(y))$. Note that $f(U)=(X-\pi^{-1}(g(y)))\cup f(\pi^{-1}(y))$. So $f(U)\cap \pi^{-1}(g(y))=f(\pi^{-1}(y))$ is open in $\pi^{-1}(g(y))$. Since $f$ is also a closed map, the set $f(\pi^{-1}(y))$ is both open and closed in the connected fibre $\pi^{-1}(g(y))$ and hence $f(\pi^{-1}(y))=\pi^{-1}(g(y))$ since $\pi$ is surjective.
\end{proof}

\begin{lemma}\label{sigma-lem} Let $\pi:X\to Y$ be a surjective morphism of two normal varieties with connected fibres. Let $f:X\to X$ and $g:Y\to Y$ be two surjective endomorphisms such that $g\circ \pi=\pi\circ f$.
Let $m\ge 0$ be an integer and $\Sigma_m:=\{Z\subseteq Y\,|\, Z \,\text{ is irreducible closed},\, \dim(Z)=m,\, \text{ and }\, \pi^{-1}(Z)\, \text{ is reducible \,}\}$.
Then $g^{-1}(\Sigma_m)\subseteq \Sigma_m$.
\end{lemma}
\begin{proof} Let $Z'$ be an $m$-dimensional irreducible closed subvariety of $Y$ such that $Z:=g(Z')\in \Sigma_m$.
Then $f|_{\pi^{-1}(Z')}:\pi^{-1}(Z')\to \pi^{-1}(Z)$ is surjective by Lemma \ref{surj-lem}.
Since $\pi^{-1}(Z)$ is reducible, so is $\pi^{-1}(Z')$.
\end{proof}

\begin{lemma}\label{lem-inv-des} Let $\pi:X\to Y$ be a surjective proper morphism of two normal varieties with connected fibres. Let $f:X\to X$ and $g:Y\to Y$ be two surjective endomorphisms such that $g\circ \pi=\pi\circ f$.
Suppose $E$ is a closed subset of $X$ such that $f^{-i}(E)=E$ for some $i>0$.
Then $g^{-j}(\pi(E))=\pi(E)$ for some $j>0$.
\end{lemma}
\begin{proof} It suffices for us to consider the case when $E$ is irreducible and $f^{-1}(E)=E$ after replacing $f$ by a positive power.
Let $F:=\pi(E)$.
Since $\pi$ is proper, $F$ is irreducible closed in $Y$.
Let $F'$ be any irreducible closed subset of $Y$ such that $g(F')=F$.
By Lemma \ref{surj-lem}, $E\subseteq f(\pi^{-1}(F'))$.
So $E= f(\pi^{-1}(F'))\cap f(f^{-1}(E))=f(\pi^{-1}(F')\cap f^{-1}(E))$.
Then $\pi^{-1}(F')\cap E= E$ and hence $F=\pi(E)\subseteq F'$.
Since $F'$ is irreducible and $\dim(F')=\dim(F)$, $F'=F$.
Note that $g(F)=F$.
So the lemma is proved.
\end{proof}

\section{$Q$-abelian case}\label{sec-qab}

In the case of characteristic $0$, the following result was proved in \cite[Lemma 2.12]{Na-Zh}. It still holds
for any characteristic, with a modified proof below.

\begin{lemma}\label{lem-ac}
Let $X$ be a $Q$-abelian variety.
Then there is a quasi-\'etale cover $\pi_A:A\to X$ such that the following hold.
\begin{itemize}
\item[(1)] $A$ is an abelian variety.
\item[(2)] $\pi_A$ is Galois.
\item[(3)] If there is another quasi-\'etale cover $\pi_B:B\to X$ from an abelian variety $B$,
then there is an \'etale morphism $\tau:B\to A$ such that $\pi_B=\pi_A\circ \tau$.
\end{itemize}
We call $\pi_A$ the {\bf Albanese closure} of $X$ in codimension one.
\end{lemma}

\begin{proof}
Since $X$ is $Q$-abelian, there is a quasi-\'etale morphism $\pi_A:A\to X$ where $A$ is an abelian variety. Let $p:A'\to A$ be the Galois closure of $\pi_A$. Then $p$ is quasi-\'etale and hence \'etale by the purity of branch loci. Note that $A'$ is an abelian variety (cf.~\cite[Chapter IV, 18]{Mu}). So we may assume $\pi_A$ is Galois.
Then $X\cong A/G_A$, where $G_A$ is a finite subgroup of $\Aut_{var}(A)$ (the automorphism group of the variety $A$).
Let $G_0=G_A\cap (translations\, on\, A)$.
Then $A/G_0\to X$ is quasi-\'etale and Galois and $A/G_0$ is an abelian variety.
So we may assume $G_A$ contains no translation.
We next show that $\pi_A$ satisfies the universal property.

Suppose there is another quasi-\'etale cover $\pi_B:B\to X$ from an abelian variety $B$.
By taking the base change and the Galois closure, there exist \'etale morphisms $C\to A$ and $C\to B$ over $X$ such that the composition $C\to X$ is Galois.
Clearly, $C$ is an abelian variety.
Then $X\cong C/G_C$ and $A\cong C/H_A$ where $G_C$ is a finite subgroup of $\Aut_{var}(C)$ and $H_A=\Gal(C/A)$ is a subgroup of $G_C$.
Similarly, $B\cong C/H_B$ where $H_B=\Gal(C/B)$ is a subgroup of $G_C$.
Since $A$ and $B$ are both abelian varieties, both $H_A$ and $H_B$ are translation subgroups of $C$.
By our construction of $A$, $H_A=G_C\cap (translations\, on\, C)$.
So $H_B$ is a subgroup of $H_A$.
Hence there is a natural \'etale quotient $\tau:B\to A=B/(H_A/H_B)$ such that $\pi_B=\pi_A\circ \tau$.
\end{proof}
\begin{corollary}\label{cor-lift} Let $f:X\to X$ be a quasi-\'etale endomorphism of a $Q$-abelian variety $X$. Then there are a quasi-\'etale morphism $\pi_A:A\to X$ and an \'etale endomorphism $f_A:A\to A$ such that $\pi_A\circ f_A=f\circ\pi_A$.
\end{corollary}

\begin{proof}
We denote by $f: X_1 = X \to X_2 = X$.
Let $\pi_i : A_i \to X_i$ be the Albanese closure.
Note that $f\circ\pi_1$ is quasi-\'etale. By Lemma \ref{lem-ac}, such $f_A$ exists.
\end{proof}

\begin{corollary}\label{cor-qab} Let $f:X\to X$ be a polarized separable endomorphism of a $Q$-abelian variety $X$. Then there is no $f^{-1}$-periodic proper subvariety of $X$.
\end{corollary}
\begin{proof} Since $f$ is separable, $f$ is quasi-\'etale by the ramification divisor formula.
Let $\pi_A:A\to X$ be the Albanese closure and $f_A$ the lifting of $f$ which is \'etale by Corollary \ref{cor-lift}.
Note that $f_A$ is polarized by Theorem \ref{polar_des} and  Theorem \ref{thm-num-lin}.
By \cite[Lemmas 3.2 and 3.3]{MZ}, there is no $f_A^{-1}$-periodic proper subvariety of $A$.
So we are done.
\end{proof}

Let $\pi:X\to Y$ be a surjective morphism of normal varieties.
For any point $y\in Y$ (not necessarily a closed point),
denote by $X_y$ the scheme-theoretical fibre of $\pi$ over $y$.
Let $\eta$ be the generic point of $Y$.
Denote by $X_\eta$ the generic fibre and $X_{\overline{\eta}}:=X_\eta\times_{k(Y)} \overline{k(Y)}$ the geometric generic fibre where $k(Y)$ is the function field of $Y$ and $\overline{k(Y)}$ is the algebraic closure of $k(Y)$.

\begin{lemma}\label{lem-gen-fib-irr} Let $\pi:X\to Y$ be a surjective projective morphism of normal varieties such that $\pi_\ast\mathcal{O}_X=\mathcal{O}_Y$.
Then the general fibre $X_y$ is irreducible.
\end{lemma}
\begin{proof} Let $X_{\eta}$ be the generic fibre.
Let $k(Y)^{\sep}$ be the separable closure of $k(Y)$ in $\overline{k(Y)}$.
Denote by $X_{\eta^s}:=X_{\eta}\otimes_{k(Y)}k(Y)^{\sep}$.
Since $\eta^s:=\SSpec k(Y)^{\sep}\to Y$ is flat, $H^0(X_{\eta^s},\mathcal{O}_{X_{\eta^s}})=H^0(\eta^s,\mathcal{O}_{\eta^s})=k(Y)^{\sep}$ by \cite[Chapter III, Proposition 9.3]{Har}.
In particular, $X_{\eta^s}$ is connected.
Since $k(Y)^{\sep}/k(Y)$ is a limit of \'etale extensions, so is $X_{\eta^s}\to X_{\eta}$.
Therefore, $X_{\eta^s}$ is normal since $X_{\eta}$ is normal.
Hence, $X_{\eta^s}$ is connected and normal, so integral.
Since $\overline{k(Y)}/k(Y)^{\sep}$ is purely inseparable, $X_{\overline{\eta}}\to X_{\eta^s}$ is a homeomorphism.
Therefore, $X_{\overline{\eta}}$ is irreducible.
\end{proof}

\begin{lemma}\label{prop-irr} Let $\pi:X\to Y$ be a surjective projective morphism of normal varieties such that $\pi_\ast\mathcal{O}_X=\mathcal{O}_Y$.
Let $f:X\to X$ and $g:Y\to Y$ be two polarized separable endomorphisms such that $g\circ \pi=\pi\circ f$.
Suppose $Y$ is $Q$-abelian.
Then all the fibres of $\pi$ are irreducible.
\end{lemma}
\begin{proof} By Lemma \ref{lem-gen-fib-irr}, the general fibre of $\pi$ is irreducible.
Then $\overline{\Sigma_0}\neq Y$ and $g^{-1}(\overline{\Sigma_0})=\overline{g^{-1}(\Sigma_0)}\subseteq\overline{\Sigma_0}$ by Lemmas \ref{open-lem} and \ref{sigma-lem} and the notation there.
Since $\overline{\Sigma_0}$ is closed, $g^{-1}(\overline{\Sigma_0})=\overline{\Sigma_0}$.
So $\Sigma_0=\emptyset$ by Corollary \ref{cor-qab}.
\end{proof}

\begin{lemma}\label{fibres-rc+irr}
Let $\pi:X\to Y$ be a surjective projective morphism of normal projective varieties such that $\pi_\ast\mathcal{O}_X=\mathcal{O}_Y$.
Let $f:X\to X$ and $g:Y\to Y$ be two polarized separable endomorphisms such that $g\circ\pi=\pi\circ f$.
Suppose that $Y$ is $Q$-abelian.
Then $\pi$ is equi-dimensional.
\end{lemma}

\begin{proof}
By Lemma \ref{prop-irr}, all the fibres of $\pi$ are irreducible.
Let $d:=\dim(X)-\dim(Y)$ and
$\Sigma:=\{y\in Y\,|\, \dim(X_y)> d\}$.
Then $\pi$ is equi-dimensional outside $\Sigma$.
By Lemma \ref{sigma-lem}, $X_{g(y)}=f(X_y)$ set theoretically.
Since $f$ is finite surjective, $g^{-1}(\Sigma)= \Sigma$.
By Corollary \ref{cor-qab}, $\Sigma=\emptyset$.
\end{proof}

In the case of characteristic $0$, any dominant rational map from a normal projective variety with rational singularities to an abelian variety is a morphism.
In the case of positive characteristic, we have the following result which follows from \cite[Theorem 4.8]{GNT}.

\begin{lemma}\label{lem-alb} Let $X$ be a normal projective variety of dimension $\le 3$ over
the field $k$ with $char\, k>5$.
Suppose $(X,\Delta)$ is a klt pair for some effective Weil $\mathbb{Q}$-divisor $\Delta$.
Let $\pi:X\dashrightarrow A$ be a rational map to an abelian variety $A$.
Then $\pi$ is a morphism.
\end{lemma}

\begin{proof} Let $p:X'\to X$ be a resolution of $\pi$.
Taking a product, we may assume $\dim(X) = 3$.
By \cite[Theorem 4.8]{GNT}, all the fibres of $\pi$ are rationally chain connected and hence $p$ is a morphism
since an abelian variety contains no rational curves and by the rigidity (cf. \cite[Lemma 1.15]{De}).
\end{proof}

With the above lemma, we can show the following by using almost the same proof as \cite[Lemma 5.3]{MZ}.
We rewrite it here for the reader's convenience.

\begin{lemma}\label{mor-q-abelian} Let $X$ be a normal projective
variety of dimension $\le 3$ over the field $k$ with $char\, k>5$.
Suppose $(X,\Delta)$ is a klt pair for some effective Weil $\mathbb{Q}$-divisor $\Delta$.
Let $\pi:X\dashrightarrow Y$ be a dominant rational map where $Y$ is $Q$-abelian.
Suppose further that the normalization of the graph $\Gamma_{X/Y}$ is equi-dimensional over $Y$
(this holds when $k(Y)$ is algebraically closed in $k(X)$, $f: X \to X$ is polarized separable and
$f$ descends to some polarized separable $f_Y : Y \to Y$; see Lemma \ref{fibres-rc+irr}).
Then $\pi$ is a morphism.
\end{lemma}

\begin{proof} Let $W$ be the normalization of the graph $\Gamma_{X/Y}$ and $p_1:W\to X$ and $p_2:W\to Y$ the two projections.
Let $\tau_1:A\to Y$ be a finite surjective morphism \'etale in codimension one with $A$ an abelian variety.
Let $W'$ be an irreducible component of the normalization of $W\times_Y A$ which dominates $W$ and $\tau_2:W'\to W$ and $p_2':W'\to A$ the two projections.
Taking the Stein factorization of the composition $W'\to W\to X$, we get a birational morphism $p_1':W'\to X'$ and a finite morphism $\tau_3:X'\to X$.
$$\xymatrix{
X'\ar[d]^{\tau_3}         & W'\ar[l]_{p_1'}\ar[d]^{\tau_2}\ar[r]^{p_2'} &A\ar[d]^{\tau_1}\\
X                & W\ar[l]_{p_1}\ar[r]^{p_2}     &Y}$$

Since $p_2$ is equi-dimensional, by the base change, $\tau_2$ is \'etale in codimension one.
Let $U\subseteq X$ be the domain of $p_1^{-1}:X\dashrightarrow W$.
Then, $\Codim(X-U) \ge 2$, and the restriction
$\tau_3^{-1}(U)\to U$ of $\tau_3$ is \'etale in codimension one, since so is $\tau_2$.
Therefore, $\tau_3$ is \'etale in codimension one. In particular, by the ramification divisor formula, $K_{X'}+\Delta'=\tau_3^\ast(K_X+\Delta)$ with $\Delta'=\tau_3^\ast \Delta$ an effective Weil $\Q$-divisor.
Since $(X,\Delta)$ is klt and
$\Delta' \ge 0$,
$(X',\Delta')$ is klt by Lemma \ref{lem-km-p}.
Clearly, $\pi':=p_2'\circ p_1'^{-1}:X'\dashrightarrow A$ is a dominant rational map, since $p_1'$ is birational and $p_2'$ is surjective.
Then $\pi'$ is a surjective morphism (with $p_2'=\pi'\circ p_1'$) by Lemma \ref{lem-alb}.
Suppose $\pi$ is not defined over some closed point $x\in X$. Then $\dim(W_x)>0$ with $W_x=p_1^{-1}(x)$ and $\dim(p_2(W_x))>0$ by the rigidity (cf. \cite[Lemma 1.15]{De}).
Hence, $\dim(p_2'(\tau_2^{-1}(W_x)))>0$ and then $\dim(p_1'(\tau_2^{-1}(W_x)))>0$.
However, $p_1'(\tau_2^{-1}(W_x))=\tau_3^{-1}(x)$ has only finitely many points.
This is a contradiction.
\end{proof}

\section{Global index-1 cover}

We follow \cite[Section 3.2]{ENS} or \cite[Definition 5.19]{KM} and consider the {\it global index-$1$ cover} of a normal projective variety in arbitrary characteristic $p \ge 0$.
Let $X$ be a normal projective variety and $D$ a $\Q$-Cartier integral divisor with $D\sim_{\Q}0$ ($\Q$-linear equivalence).
Define the {\it global index} $r$ of $D$ to be the minimal positive integer $r$
with $rD\sim 0$ (linear equivalence).
Pick a non-zero $u\in H^0(X,rD)$, one may define a cyclic covering
$$\pi \, : \, V: \, = \, V(D,r,u) \, = \, \SSpec \bigoplus_{i=0}^{r-1} \mathcal{O}_X(-iD)\, \to \, X .$$
By the construction, $V$ can be locally regarded as a hypersurface of $X\times \mathbb{A}^1$. So $V$ is $S_2$ (cf.~\cite[Proposition 5.3]{KM}).
Since $H^0(X,rD)\cong k$, for any two non-zero $u,v\in H^0(X,rD)$, $u=a^rv$ for some $a\in k$. In particular, $\pi$ is independent of the choice of $u$ up to isomorphism.
Let $f:X\to X$ be a surjective endomorphism such that $f^*D\sim D$.
Then $f$ lifts to an endomorphism on $V$;
see \cite[Lemma 3.2.5]{ENS} or proof of \cite[Proposition 3.5]{Na-Zh}.
Note that $\pi$ here is not separable if $p \, | \, r$.
If $D=K_X$, we call $\pi$ the {\it global index-$1$ cover} of $X$.

\begin{lemma}\label{lem-lift-e} Let $\pi:X\to Y$ be a purely inseparable finite surjective morphism between two normal varieties and let $f : X \to X$ and $g : Y \to Y$ be two surjective endomorphisms such that $g\circ \pi=\pi\circ f$.
Suppose $g$ is quasi-\'etale.
Then $f$ is quasi-\'etale.
\end{lemma}

\begin{proof} Let $W$ be the fibre product of $\pi$ and $g$. Let $W_1$ be an irreducible component of $W$ dominating $X$ and $Y$.
Denote by $p_1:W_1\to X$ and $p_2:W_1\to Y$ the two projections.
We may assume there is a finite surjective morphism $\tau:X\to W_1$ such that $p_1\circ \tau=f$ and $p_2\circ \tau=\pi$.
Since $g$ is quasi-\'etale, $g$ is separable and hence $f$ is separable (cf.~ Lemma \ref{sepdeg}).
Therefore, $\tau$ is separable.
Since $\pi$ is purely inseparable, $\tau$ is also purely inseparable.
So $\deg \tau=1$ and $\tau$ is the normalization of $W_1$.
Note that $p_1$ is also quasi-\'etale and $W_1$ is smooth in codimension $1$.
Then $f$ is quasi-\'etale.
\end{proof}

\begin{lemma}\label{lem-index} Let $f:X\to X$ be a $q$-polarized separable endomorphism of a normal projective variety with $K_X\sim_{\Q}0$.
Let $\pi:V\to X$ be the global index-$1$ cover of $X$.
Then $V$ is a normal projective variety with $K_V\sim 0$.
\end{lemma}

\begin{proof}  We may assume the characteristic $p>0$.
Let the global index of $K_X$ be $r=ap^m$ with $(a,p)=1$ and $m\ge 0$.
Since $K_X\sim_{\Q}0$, $f$ is quasi-\'etale by Lemma \ref{lem-pe-qe}.
So $f^*K_X\sim K_X$.
Let $D=p^mK_X$.
Then $f^*D\sim D$ and the global index of $D$ is $a$.
Let $\pi_1:X_1\to X$ be the cyclic cover induced by $D$.
Then $\pi_1$ is quasi-\'etale and hence $X_1$ is smooth in codimension $1$.
Since $X_1$ is $S_2$, $X_1$ is normal.
Note that $\pi$ factors through $\pi_1$.
So we may assume $r=p^m$ and $\pi$ is purely inseparable.

Let $\tau:W\to V$ be the normalization and $C$ the conductor of $\tau$ (cf.~\cite[arxiv version 1, Remark 5.2]{Na-Zh}).
Denote by $\eta:=\pi\circ \tau$.
Let $h:W\to W$ be the lifting of $f$ which is also $q$-polarized separable (cf.~Theorem \ref{polar_des} and  Lemma \ref{sepdeg}).
Note that $\omega_Y\cong \mathcal{O}_Y$ (cf.~\cite[Proposition 5.68]{KM}).
Then $\omega_W(C)=\tau^*\omega_V\cong \mathcal{O}_W$.
Since $\eta$ is purely inseparable and $f$ is quasi-\'etale, $h$ is still quasi-\'etale by Lemma \ref{lem-lift-e}.
Then $K_W\equiv 0$ by Lemma \ref{lem-pe-qe} and hence $C=0$ since $C$ is effective.
So $V$ itself, being $S_2$, is normal and $K_V\sim 0$ (cf.~\cite[Proposition 5.75]{KM}).
\end{proof}

\begin{lemma}\label{lem-index-tame} Let $f:X\to X$ be a $q$-polarized separable endomorphism of a normal projective variety with $K_X\sim_{\Q}0$.
Let $\pi:V\to X$ be the global index-$1$ cover of $X$ and $f_V:V\to V$ the lifting of $f$.
Suppose $\deg f^{\Gal}$ is co-prime to $p$.
Then so is $\deg f_V^{\Gal}$.
\end{lemma}
\begin{proof} Let $\pi_1:X_1\to X$ be the cyclic cover as in the proof of Lemma \ref{lem-index}.
Since our base field is algebraically closed, $\pi_1$ is Galois with degree co-prime to $p$.
Let $f_{X_1}:X_1\to X_1$ be the lifting of $f$.
Then $\deg (\pi_1\circ f_{X_1})^{\Gal}=\deg (f\circ\pi_1)^{\Gal}$ is co-prime to $p$ by Lemma \ref{lem-gal-gen}.
In particular, $\deg f_{X_1}^{\Gal}$ is co-prime to $p$.
So we may assume $\pi$ is purely inseparable.
Now the lemma follows from Lemma \ref{lem-tame-des}.
\end{proof}

\section{Surface case and the proof of Theorem \ref{thm-Q-surf}}

In this section, we consider the surjective endomorphism $f:X\to X$ of a normal algebraic surface over the field $k$ of characteristic $p>0$ with $\deg f^{\Gal}$ co-prime to $p$, so that \cite[Proposition 5.20]{KM} can be applied in positive characteristic (cf.~Lemma \ref{lem-km-p}).
Note that replacing $f$ by any positive power $f^s$, $\deg (f^s)^{\Gal}$ is still co-prime to $p$ (cf.~Lemma \ref{lem-gal-gen}).

Let $(X,\Delta)$ be a normal algebraic surface pair with $\Delta$ a Weil $\Q$-divisor on $X$.
Fix $u\in X$ a closed point.
Let $\pi:\widetilde{X}\to X$ be a resolution at $u$ with exceptional divisor $E=\sum E_i$.
Then $K_{\widetilde{X}}+\pi_*^{-1}\Delta\equiv_{\pi} \sum a_i E_i$, where $\pi_*^{-1}\Delta$ is the strict transform of $\Delta$ and $\equiv_\pi$ is the numerical equivalence over $X$.
If for any resolution at $u$, $a_i\ge -1$ (resp.~$a_i>-1$), then we say $(X,\Delta)$ is {\it numerically sub-lc} (resp.~{\it numerically sub-klt}) at $u$.
We say $(X,\Delta)$ is {\it numerically lc} if further $\Delta$ is effective.
It is known that $(X,\Delta)$ is numerically lc if and only if $(X,\Delta)$ is lc; see \cite[Proposition 6.3]{FT}.

Denote by $\Nlc(X)$ (resp.~$\Nklt(X)$) the finite subset of $X$ where $X$ is
non-numerically sub-lc (resp.~ non-numerically sub-klt).

\begin{lemma}\label{lem-nonklt}
Let $X$ be a normal algebraic surface over the field $k$ of characteristic $p>0$.
Let $f:X\to X$ be a surjective endomorphism of $X$ with $\deg f^{\Gal}$ co-prime to $p$.
Then:
\begin{itemize}
\item[(1)] Both $\Nlc(X)$ and $\Nklt(X)$ are  $f^{-1}$-invariant.
\item[(2)] $\Nklt(X)\cap \Supp R_f=\emptyset$.
\end{itemize}
\end{lemma}
\begin{proof}
(1) Since $f$ is separable, by the ramification divisor formula, $K_X-R_f=f^*K_X$.
Let $u\in X$ and $v=f(u)$ such that $v\in \Nlc(X)$. By Lemma \ref{lem-km-p}, $(X, -R_f)$ is non-numerically sub-lc  at $u$ and hence so is $X$ since $R_f$ is effective.
Therefore, $f^{-1}(\Nlc(X))\subseteq \Nlc(X)$.
Since $\Nlc(X)$ is a finite set, $f^{-1}(\Nlc(X))= \Nlc(X)$.
The same proof works for $\Nklt(X)$.

(2) If $\Nklt(X)\cap \Supp R_f\neq \emptyset$, then $\Nklt(X)\cap \Supp R_{f^i}\neq \emptyset$ for any $i>0$.
So replacing $f$ by a positive power, we may assume $f^{-1}(u)=u$ for any $u\in \Nklt(X)$ by (1).
Suppose there exists $u\in \Nklt(X)\cap \Supp R_f$.
Let $D=\sum d_k D_k$ be an effective Weil $\Q$-divisor of $X$ where $D_k$ is a prime divisor containing $u$.
Denote by $m_D:=\sum d_k$.
Let $$S:=\{D=\sum d_k D_k\ge 0\,|\, (X,-D) \,\text{is non-numerically sub-klt at} \, u\,\text{and}\,u\in D_k, \forall k  \}.$$
Clearly $0\in S\neq \emptyset$.
Denote by $m:=\sup\{m_D\,|\, D\in S\}\ge 0.$
We claim that $m<\infty$.

Let $\pi: \widetilde{X}\to X$ be the minimal resolution of the point $u$ with $E=\sum_{i=1}^{n}E_i$ the exceptional divisor. Denote by $A$ the intersection matrix $(E_i\cdot E_j)_{i,j}$.
Let $K_{\widetilde{X}}\equiv_{\pi}\sum a_i E_i$.
For any prime divisor $P$ containing $u$, $\pi_*^{-1}P\equiv_\pi \sum b_i E_i$ with $b_i<0$.
Note that $\pi_*^{-1}P\cdot E_i\in \mathbb{Z}$. So $A(b_1,\cdots,b_n)^{T}\in \mathbb{Z}^n$ and hence $(b_1,\cdots,b_n)\in A^{-1}(\mathbb{Z}^n)=\frac{A^*}{\det(A)}(\mathbb{Z}^n)$.
Since all the entries of $A$ are integers, $db_i\in \mathbb{Z}$ where $d=|\det(A)|$.
Therefore, $b_i\le -\frac{1}{d}$.
Let $D=\sum_k d_k D_k$ be an effective Weil $\Q$-divisor of $X$ where $D_k$ is a prime divisor containing $u$ for each $k$.
Then $\pi_*^{-1}D=\sum_k d_k\pi_*^{-1}D_k\equiv_{\pi} \sum_i c_i E_i$ with $c_i\le -\sum_k \frac{d_k}{d}=-\frac{m_D}{d}$.
Suppose $m_D\ge -a_i d$ for each $i$.
Then  $K_{\widetilde{X}}-\pi_*^{-1}D\equiv_{\pi} \sum_i (a_i-c_i) E_i$.
Since $a_i-c_i\ge 0$ for each $i$, $(X,-D)$ is numerically sub-klt at $u$ and hence $D\not\in S$.
Therefore for any $D\in S$, $m_D<\max_i\{-a_i d\}$.

Let $D\in S$ such that $m_D>m-1$.
Since $K_X-R_f-f^*D=f^*(K_X-D)$ and by Lemma \ref{lem-km-p}, we have $(X,-R_f-f^*D)$ is non-numerically sub-klt at $u$.
Write $R_f=B+C$ where each prime divisor of $B$ contains $u$ and $C$ does not contain $u$.
Then $(X,-B-f^*D)$ is non-numerically sub-klt at $u$.
Note that each prime divisor of $f^*D$ still contains $u$.
Then $B+f^*D\in S$.
Since $B\neq 0$, $m\ge m_{B+f^*D}\ge m_B+m_D\ge 1+m_D>m$.
So we get a contradiction.
\end{proof}

If $f$ is further non-isomorphic, the proof of Wahl \cite[Theorem 2.8]{Wa} also works in any characteristic.
\begin{theorem}\label{thm-w} Let $X$ be a normal algebraic surface over the field $k$ of characteristic $p>0$.
Let $f:X\to X$ be a non-isomorphic surjective endomorphism of $X$ with $\deg f^{\Gal}$ co-prime to $p$.
Then $X$ is lc.
\end{theorem}

\begin{proof}
By \cite[Proposition 6.3]{FT}, we only need to show that $X$ is numerically lc.
Let $Z:=\Nklt(X)$. By Lemma \ref{lem-nonklt}, $Z$ is $f^{-1}$-invariant and $Z\cap \Supp R_f=\emptyset$.

Let $x_0\in Z$, $V=X\backslash((\Sing(X)\backslash\{x_0\})\cup f(\Supp R_f))$ and $U=f^{-1}(V)$. Replacing $f$ by a positive power, we may assume $f^{-1}(x_0)=x_0$. Then we have a quasi-\'etale morphism $f|_U:(U,x_0)\to (V,x_0)$ which is \'etale off $x_0$.
Let $\pi:\widetilde{V}\to V$ be the log resolution of $V$ and $E$ the reduced $\pi$-exceptional divisor. In particular, $(\widetilde{V},E)$ is lc.
Let $U'$ be the normalization of the fibre product $\widetilde{V}\times_V U$. Denote by $f':U'\to \widetilde{V}$ and $\pi':U'\to U$ the two projections.
Let $p:\widetilde{U}\to U'$ be the minimal resolution of $U'$ and $\widetilde{f}:\widetilde{U}\to \widetilde{V}$ and $\widetilde{\pi}:\widetilde{U}\to U$ the two induced projections.
Let $E'$ be the reduced $\pi'$-exceptional divisor, $\widetilde{E}$ the strict transform of $E'$ on $\widetilde{U}$, and $\widetilde{F}$ the reduced $p$-exceptional divisor.

Since $f'$ is \'etale off $E$, $f'^*(K_{\widetilde{V}}+E)=K_{U'}+E'$ by the (tame) ramification divisor formula. Note that $(U',E')$ is lc by Lemma \ref{lem-km-p}.
Then $\widetilde{f}^*(K_{\widetilde{V}}+E)=p^*(K_{U'}+E')=K_{\widetilde{U}}+\widetilde{E}+\sum a_i \widetilde{F}_i$ where $\widetilde{F_i}$ is a $p$-exceptional prime divisor and $a_i=0$ or $1$.
Taking the relative Zariski decomposition, we write $K_{\widetilde{U}}+\widetilde{E}+\widetilde{F}=P_{\widetilde{U}/U}+N_{\widetilde{U}/U}$ and $K_{\widetilde{V}}+E=P_{\widetilde{V}/V}+N_{\widetilde{V}/V}$.
Then $P_{\widetilde{U}/U}=\widetilde{f}^*P_{\widetilde{V}/V}$ and $N_{\widetilde{U}/U}=\widetilde{f}^*P_{\widetilde{V}/V}+\sum (1-a_i) \widetilde{F}_i$.
In particular, $P_{\widetilde{U}/U}^2=(\deg f) P_{\widetilde{V}/V}^2=(\deg f) P_{\widetilde{U}/U}^2$.
Since $\deg f>1$, $P_{\widetilde{U}/U}^2=0$ and hence $U$ is lc.
\end{proof}

Let $X$ be a variety (which is not necessarily smooth, compact, or irreducible) over the field $k$ of characteristic $p>0$ and let $H^i(X, \mathbb{Z}_{\ell})$ be the $\ell$-adic cohomology group of $X$.
Set $H^i(X,\mathbb{Q}_{\ell}) := H^i(X, \mathbb{Z}_{\ell})\otimes_{\mathbb{Z}_{\ell}}\mathbb{Q}_{\ell}$.
Denote by $b^i(X, \ell) := \dim_{\mathbb{Q}_{\ell}}H^i(X,\mathbb{Q}_{\ell})$ the $i$-th $\ell$-adic Betti number of $X$.
The $\ell$-adic Euler
characteristic is denoted
by $e(X, \ell) :=\sum_{i \ge 0} (-1)^i b^i(X, \ell)$.
Note that $e(X, \ell)$ is independent of the choice of $\ell$ prime to $p$.
So we simply denote by $e(X):=e(X, \ell)$.
We recall here two basic facts.
\begin{itemize}
\item[(1)] Let $Y$ be a closed subvariety of $X$ and $U=X\backslash Y$. Then $e(X)=e(U)+e(Y)$. (cf.~\cite{La}, \cite[Lemma 3.1]{Hu}).
\item[(2)] Let $f:X\to Y$ be an \'etale cover with $f^{\Gal}$ being tame. Then $e(X)=(\deg f) e(Y)$ (cf.~\cite[Proposition 3.12]{DL}).
\end{itemize}

Inspired by the proof of Nakayama \cite[Section 7.3]{ENS} in the case of characteristic $0$,
we extend it to characteristic $p > 5$.

\begin{lemma}\label{lem-nak} Let $f:X\to X$ be a quasi-\'etale non-isomorphic endomorphism of a normal projective surface $X$ over the field $k$ of characteristic $p>5$.
Suppose $\deg f^{\Gal}$ is co-prime to $p$ and $K_X\sim 0$.
Then $X$ is a $Q$-abelian surface.
\end{lemma}

\begin{proof}
Denote by $Z:=\Nklt(X)$. Then $Z$ is finite and $f^{-1}(Z)=Z$ by Lemma \ref{lem-nonklt}.
Let $\theta_k:V_k\to X$ be the Galois closure of $f^k:X\to X$ and let $\tau_k:V_k\to X$ be the induced finite Galois covering such that $\theta_k=f^k\circ \tau_k$.
By Lemma \ref{lem-gal-comp}, $\deg \theta_k$ and $\deg \tau_k$ are co-prime to $p$ for any $k$.
We claim that $U_k:=\theta_k^{-1}(X\backslash Z)$ is smooth and the Euler number $e(U_k)=0$ when $k\gg 1$.

There exist finite Galois morphisms $g_k, h_k: V_{k+1}\to V_k$ such that $\tau_k\circ g_k=\tau_{k+1}$ and $\tau_k\circ h_k=f\circ \tau_{k+1}$; see \cite[Lemma 3.3.1]{ENS} or \cite[Lemma 2.5]{Na-Zh}.
Then $\deg g_k$ and $\deg h_k$ are co-prime to $p$ for any $k$.
Since $X\backslash Z$ is strongly F-regular, thanks to the assumption that
$p > 5$ (cf.~\cite{Hara}), $g_k|_{U_{k+1}}$ and $h_k|_{U_{k+1}}$ are \'etale when $k\gg 1$ by \cite[Main Theorem]{BCGST}.
Assume from now on that $k\gg 1$.
Since $g_k$ and $h_k$ are Galois and by the purity of branch loci,
we have $g_k^{-1}(\Sing U_k)=h_k^{-1}(\Sing U_k)=\Sing U_{k+1}$.
Therefore, $\sharp \Sing U_{k+1}= (\deg g_k)\sharp \Sing U_k=(\deg h_k)\sharp \Sing U_k$. Since $\deg h_k>\deg g_k$, $\sharp \Sing U_k=0$.
Since $g_k$ and $h_k$ are tame, $e(U_{k+1})=(\deg g_k) e(U_k)=(\deg h_k) e(U_k)$ and hence $e(U_k)=0$.
So the claim is proved.

Next we claim that $Z$ is empty. Let $\delta:W\to V_k$ be the minimal resolution of $V_k$
for some $k\gg 1$.
Since $X$ is lc by Theorem \ref{thm-w} and $\theta_k$ is quasi-\'etale and tame, $V_k$ is also lc with
$\theta_k^{-1}(Z) = \Nklt(V_k)$ (cf. Lemma \ref{lem-km-p}).
Further, since $V_k$ is also Gorenstein,
$K_W=\delta^*K_{V_k} -E\sim -E$ and $E$ is equal to the full $\delta$-exceptional reduced divisor
$\delta^{-1}\theta_k^{-1}(Z)$;
precisely, a connected component of $E$ is either an elliptic curve or a nodal curve or a cycle of smooth rational curves by the classification of Gorenstein lc surfaces (cf.~\cite[Theorem 4.7]{KM}).
So $e(E)$ equals the number of rational curves contained in $E$.
Therefore, $e(E)\le \rho(W)-1$ where $\rho(W)$ is the Picard number of $W$.
Note that $e(W)=e(U_k)+e(E)=e(E)$ and $e(W)\ge 2-4\dim(\Alb(W))+\rho(W)$ where $\Alb(W)$ is the Albanese variety of $W$.
Thus $\dim(\Alb(W))\ge 1$.
So we have a ruling $\pi:W\to T$ to a smooth projective curve $T$ of genus $\ge 1$.
For a general fibre $F$ of $\pi$, we have $E\cdot F=-K_W\cdot F=2$, which implies that an irreducible component of $E$ dominates $T$.
Thus, $\dim(\Alb(W))= 1$ and $T$ is an elliptic curve, since each irreducible component of $E$ has genus $\le 1$.
If a connected component of $E$ is not an elliptic curve, then it is contained in a fibre of $\pi$;
however this is impossible,
since successfully blowing down $(-1)$-curves in fibres of $\pi$ will
reach a $\mathbb{P}^1$-bundle over $T$ whose fibres of course have no images of
nodal curves or cycles of rational curves.
Hence, $E$ is a disjoint union of elliptic curves and $e(E)=0$.
Since $\dim(\Alb(W))= 1$, $b_2(W)=2$ by the previous inequality, and therefore $W$ is a $\mathbb{P}^1$-bundle over $T$.
In particular, $E^2 = K_W^2=0$, a contradiction to the negativity of $E$.
So the claim is proved, i.e., $X$ is klt.

Hence $X$ is strongly F-regular by \cite{Hara}. Then $X$ is $Q$-abelian by Theorem \ref{thm-FQ}.
\end{proof}

\begin{proof}[Proof of Theorem \ref{thm-Q-surf}] By Theorem \ref{thm-w}, $X$ is ($\Q$-Gorenstein and) lc.
Now $K_X$ being pseudo effective implies that $K_X\sim_{\Q} 0$
by Lemma \ref{lem-pe-qe} and the Abundance Theorem for surfaces (cf.~\cite[Theorem 1.2]{Ta}).
Let $\pi:Y\to X$ be the global index-$1$ cover of $X$ and let $g:Y\to Y$ be the lifting.
Then $Y$ is normal and $K_Y\sim 0$ by Lemma \ref{lem-index}.
Note that $\deg g^{\Gal}$ is co-prime to $p$ by Lemma \ref{lem-index-tame}.
So $Y$ is further a $Q$-abelian surface by Lemma \ref{lem-nak}.
Denote by $Z:=\Nklt(X)$.
Then $Z$ is finite and $f^{-1}$-invariant by Lemma \ref{lem-nonklt}.
Thus $\pi^{-1}(Z)$ is $g^{-1}$-invariant and hence $\pi^{-1}(Z)=\emptyset$ by Corollary \ref{cor-qab}.
So $Z=\emptyset$ and $X$ is strongly F-regular by \cite{Hara}.
Then $X$ is $Q$-abelian by Theorem \ref{thm-FQ}.
\end{proof}

\begin{theorem}\label{thm-sur} Let $f:X\to X$ be a $q$-polarized separable endomorphism of a normal projective surface $X$ over the field $k$ of characteristic $p > 0$.
Suppose $X$ is either lc or $\Q$-factorial.
Then, replacing $f$ by a positive power, there exists an $f$-equivariant relative MMP over $Y$ $$X=X_1\to \cdots \to X_i \to \cdots \to X_r=Y$$ (i.e. $f=f_1$ descends to an endomorphism $f_i$ on each $X_i$), with every $X_i \to X_{i+1}$ a divisorial contraction or a Fano contraction, of a $K_{X_i}$-negative extremal ray, such that we have:
\begin{itemize}
\item[(1)] $K_Y\sim_{\Q} 0$ and $f_r$ is quasi-\'etale.
\item[(2)] For each $i$, $f_i$ is separable and $q$-polarized by some ample Cartier divisor $H_i$.
\item[(3)]
If $K_X$ is pseudo-effective, then $X=Y$.
\item[(4)]
If $K_X$ is not pseudo-effective, then $Y$ is either an elliptic curve or a point. In particular,  $f^\ast|_{\N^1(X)}$ is a scalar multiplication.
\item[(5)]
For each $i$, $X_i\to Y$ is equi-dimensional with every fibre irreducible.
\end{itemize}
Suppose further $p>5$ and $p\nmid \deg f^{\Gal}$.
Then:
\begin{itemize}
\item[(6)] $X_i$ is lc for each $i$ and $Y$ is $Q$-abelian.
\item[(7)]
If $\dim(Y)>0$, then $X_i$ is strongly F-regular (and hence klt) for each $i$.
\end{itemize}
\end{theorem}
\begin{proof} By \cite[Theorem 1.1]{Ta} and Remark \ref{rmk-emmp}, we can run $f$-equivariant MMP.
Note that the MMP ends up with $Y$ with $K_Y$ being nef.
If $\dim(Y)<2$, then (1) is trivial.
If $\dim(Y)=2$, then (1) follows from Lemma \ref{lem-pe-qe} and the Abundance Theorem for surfaces (cf.~\cite[Theorem 1.2]{Ta}).
(2) follows from Theorem \ref{polar_des}, Remark \ref{rmk-num-lin} and  Lemma \ref{sepdeg}.
If $K_X$ is pseudo-effective, then (3) follows from Lemma \ref{lem-pe-qe}.
For (4), the first part follows from (1) and (2) while the second follows from
the fact that $\N^1(X)$ is spanned by the pullback of $\N^1(Y)$
and the classes of $H_i$ which are $f^*$ eigenvectors.
(5) follows from Lemmas \ref{prop-irr} and \ref{fibres-rc+irr}.

Suppose further $p > 5$, and $p\nmid \deg f_i^{\Gal}$ for $i = 1$ and hence for all $i$
by Lemma \ref{lem-tame-des}.
(6) then follows from Theorems \ref{thm-w} and \ref{thm-Q-surf}.
By Lemma \ref{lem-nonklt}, the non-klt locus $\Nklt(X_i)$ of $X_i$ is $(f_i)^{-1}$-invariant for each $i$.
If $\dim(Y)>0$, $\Nklt(X_i)=\emptyset$ by Lemma \ref{lem-inv-des} and Corollary \ref{cor-qab}.
So $X_i$ is strongly F-regular by \cite{Hara} and since $p>5$.
So (7) is proved.
\end{proof}

\section{Proofs of Theorems \ref{scalarthm}, \ref{thm-cases} and \ref{thm-smooth-rc}}\label{MMP}

We refer to \cite{BW} and \cite{HX} for the MMP of klt $3$-fold pairs with characteristic $p > 5$.

\begin{proof}[Proof of Theorem \ref{scalarthm}]
If $K_X$ is pseudo-effective, then $K_X\sim_{\Q} 0$ by  Theorem \ref{thm-torsion} and the theorem is then trivial.

Now we assume $K_X$ is not pseudo-effective.
Then there exists a Fano contraction $X_i\to X_{i+1}$ with $\dim(X_{i+1})\le 2$ for some $i$ (cf.~\cite[Theorem 1.7]{BW}).
By Remark \ref{rmk-emmp}, replacing $f$ by a positive power, $f$ descends to a $q$-polarized endomorphism $f_{i+1}:X_{i+1}\to X_{i+1}$.
Since $X$ is $\Q$-factorial, so is $X_{i+1}$ by the MMP theory.
By Theorem \ref{thm-sur}, we may further run $f$-equivariant MMP starting from $X_{i+1}$ and end up with $Y$ such that $K_Y\sim_{\Q} 0$.
Then (1) follows from Lemma \ref{lem-pe-qe}.
(3) follows from Remark \ref{rmk-emmp}.
For (4), we just reason as in the proof of Theorem \ref{thm-sur}.

Assume further that $p\nmid \deg f_i^{\Gal}$ for $i = 1$ and hence for all $i$
by Lemma \ref{lem-tame-des}.
Since $K_X$ is not pseudo-effective, $\dim(Y)\le 2$.
By Theorem \ref{thm-Q-surf}, $Y$ is $Q$-abelian.
If $\dim(X_i)\le 2$, then $X_i$ is lc and $X_i\dashrightarrow Y$ is an equi-dimensional morphism with every fibre irreducible by Theorem \ref{thm-sur}.
If $\dim(X_i)=3$, then $X_i\dashrightarrow Y$ is still an equi-dimensional morphism with every fibre irreducible by Lemmas \ref{mor-q-abelian}, \ref{prop-irr} and \ref{fibres-rc+irr}.
This shows (6).

Suppose $\tau:X_i \to Z_i$ is a flipping contraction.
Let $A$ be an effective ample $\Q$-Cartier divisor such that $K_{X_i}+A$ is $\tau$-trivial and $(X_i, A)$ is klt.
Note that $K_{X_i}+A\sim_{\Q} \tau^*D$ for some $\Q$-Cartier divisor $D$ of $Z_i$.
Let $B:=\tau_*A$.
Then $K_{Z_i}+B=\tau_*(K_{X_i}+A)\sim_{\Q} \tau_*\tau^*D=D$.
Hence $(Z_i, B)$ is also klt.
By Lemma \ref{mor-q-abelian}, $Z_i\dashrightarrow Y$ is a well defined morphism.
Thus the MMP is relative over $Y$.
So (5) is proved.

For (7), suppose $\pi_i:X_i \dashrightarrow X_{i+1}$ is birational for some $i$.
If $\pi_i$ is divisorial with $E$ the exceptional divisor, then $E$ is $f_i^{-1}$-periodic.
So the image of $E$ in $Y$ is also $f_r^{-1}$-periodic by Lemma \ref{lem-inv-des}.
By Corollary \ref{cor-qab}, $E$ and hence $\pi_i(E)$ dominate $Y$.
Thus $\dim(Y)\le \dim(\pi_i(E))\le \dim(X_i)-2$.
Similarly, if $\pi_i$ is a flip with $E$ the exceptional locus of the flipping contraction $\tau:X_i \to Z_i$, then $\dim(Y)\le \dim(\tau(E))\le \dim(X_i)-3$.
Thus we always have $\dim(Y)\le 1$.
In particular, $f_r^*$ is a scalar multiplication and hence $f^*$ is a scalar multiplication by (4).

(8) is true if $X \dashrightarrow X_i$ is birational by the MMP theory. Otherwise, apply Theorem \ref{thm-sur}.
\end{proof}

\begin{lemma}\label{lem-abe-ef} Let $\pi:X\to Y$ be an equi-dimensional morphism from a normal projective variety $X$ to a $Q$-abelian variety $Y$ such that $\pi$ has irreducible general fibres.
Suppose $\pi_1^{\acute{e}t}(X_{\reg})$ is finite.
Then $Y$ is a point.
\end{lemma}

\begin{proof}
Let $A\to Y$ be the quasi-\'etale cover from an abelian variety $A$.
Let $X'$ be the normalization of $X\times_Y A$.
Since $\pi$ has irreducible general fibres, by the base change, $X'$ is irreducible.
Since $\pi$ is equi-dimensional, $X'\to X$ is also quasi-\'etale.
Hence $\pi_1^{\acute{e}t}(X'_{\reg})$ is also finite.
Suppose $\dim(A)>0$.
Let $m_A : A'=A \to A$ be the multiplication map with $m$ co-prime to $p = char\, k$
and $m > |\pi_1^{\acute{e}t}(X'_{\reg})|$.
By the base change, we have an \'etale cover $X''\to X'$ of degree $m$, a contradiction.
So $A=Y$ is a point.
\end{proof}

\begin{proof}[Proof of Theorem \ref{thm-cases}]
Let $\pi:X\to Y$ be the induced morphism and $g:=f_r:Y\to Y$ the induced $q$-polarized endomorphism as in Theorem \ref{scalarthm}.
If $\dim(X)<3$ or $Y$ is not a $Q$-abelian surface, then either Case (1) or (2) occurs by Theorem \ref{scalarthm}.
Suppose $Y$ is a $Q$-abelian surface and $\dim(X)=3$.
By Theorem \ref{scalarthm}, the MMP has only one step: a Fano contraction of $K_X$-negative extremal ray.
This is Case (3).

Suppose
$\pi_1^{\acute{e}t}(X_{\reg})$ is finite.
Since $\lim\limits_{s\to \infty}\deg f^s =\infty$, $f$ is not quasi-\'etale by the purity of branch loci.
So $K_X$ is not pseudo-effective by Lemma \ref{lem-pe-qe}.
By Theorem \ref{scalarthm}, $\pi$ is an equi-dimensional surjective morphism such that all the fibres are irreducible and $Y$ is $Q$-abelian.
So
the $Y$ in Theorem \ref{scalarthm} is a point by Lemma \ref{lem-abe-ef}.
\end{proof}

\begin{proof}[Proof of Theorem \ref{thm-smooth-rc}]
Since $\Alb(X)$ contains no rational curves and is generated by the image of the rationally chain connected $X$,
it is trivial. So (1) follows from Theorem \ref{thm-pe}.
It is known that the \'etale fundamental group of smooth and rationally chain connected projective variety is finite (cf.~\cite[Theorem 1.6]{Ko93}).
Hence (2) and  (3) follow from Theorem \ref{thm-cases}.
\end{proof}

\end{document}